\documentclass{amsart}

\usepackage{amssymb}
\usepackage{amsmath}
\usepackage{amsfonts}
\usepackage[colorlinks,linkcolor=blue]{hyperref}
\usepackage[numbers,sort&compress]{natbib}

\newtheorem{theorem}{Theorem}
\theoremstyle{plain}

\newtheorem{corollary}{Corollary}

\newtheorem{definition}{Definition}

\newtheorem{lemma}{Lemma}
\newtheorem{Hypothesis}{Hypothesis}
\newtheorem{notation}{Notation}
\newtheorem{Inverse problem}{Inverse Problem}
\newtheorem{proposition}{Proposition}
\newtheorem{remark}{Remark}

\numberwithin{equation}{section}

\begin{document}
\title[]{Inverse spectral problems for non-self-adjoint Sturm-Liouville
operators with discontinuous boundary conditions}
\author{Jun Yan and Guoliang Shi}
\subjclass[2010]{Primary 34A55; Secondary 34L40, 34L20}
\address{School of Mathematics, Tianjin University, Tianjin, 300354,
People's Republic of China}
\keywords{diffusions, eigenvalues, non-self-adjoint, multiplicity}
\email{jun.yan@tju.edu.cn}
\address{School of Mathematics, Tianjin University, Tianjin, 300354,
People's Republic of China}
\email{glshi@tju.edu.cn}
\date{\today }
\keywords{Non-selfadjoint Sturm-Liouville operators, inverse problem,
eigenvalue, norming constant}

\begin{abstract}
This paper deals with the inverse spectral problem for a non-self-adjoint
Sturm-Liouville operator with discontinuous conditions inside the interval.
We obtain that if the potential $q$ is known a priori on a subinterval $%
\left[ b,\pi \right] $ with $b\in \left( d,\pi \right] $ or $b=d$, then $h,$
$\beta ,$ $\gamma \ $and $q$ on $\left[ 0,\pi \right] \ $can be uniquely
determined by partial spectral data consisting of a sequence of eigenvalues
and a subsequence of the corresponding generalized normalizing constants or
a subsequence of the pairs of eigenvalues and the corresponding generalized
ratios. For the case $b\in \left( 0,d\right) ,$ a similar statement holds if $%
\beta ,$ $\gamma \ $are also known a priori. Moreover, if $q$ satisfies a
local smoothness condition, we provide an alternative approach instead of
using the high-energy asymptotic expansion of the Weyl $m$-function to solve
the problem of missing eigenvalues and norming constants.
\end{abstract}

\maketitle

\section{Introduction}

In this paper, we consider the non-self-adjoint Sturm-Liouville operator $%
L:=L\left( q,h,H,\beta ,\gamma ,d\right) $ defined by%
\begin{equation}
\ell y:=-y^{\prime \prime }+q\left( x\right) y  \label{ly}
\end{equation}%
on the interval $\left( 0,\pi \right) \ $with the boundary conditions%
\begin{equation}
U\left( y\right) :=y^{\prime }\left( 0\right) -hy\left( 0\right) =0,V\left(
y\right) :=y^{\prime }\left( \pi \right) +Hy\left( \pi \right) =0
\label{BCs}
\end{equation}%
and the discontinuous conditions
\begin{equation}
y\left( d+0\right) =\beta y\left( d-0\right) \text{, }y^{\prime }\left(
d+0\right) =\beta ^{-1}y^{\prime }\left( d-0\right) +\gamma y\left(
d-0\right) \text{,}  \label{DCs}
\end{equation}%
where\ $q\in L_{%
\mathbb{C}
}^{1}\left[ 0,\pi \right] $ is complex-valued, $h,$ $H\in
\mathbb{C}
\cup \left\{ \infty \right\} ,$ {$\gamma \in
\mathbb{C}
$ }and {{$\beta \in
\mathbb{R}
$, $\beta >0$}}. Note that, in an obvious notation, $h=\infty $ and $%
H=\infty $ single out the Dirichlet boundary conditions%
\begin{equation*}
U^{\infty }\left( y\right) :=y\left( 0\right) =0\text{ and }V^{\infty
}\left( y\right) :=y\left( \pi \right) =0\text{,}
\end{equation*}%
respectively. One notes that in the special case $\beta =1,$ $\gamma =0,$
the operator $L$ reduces to the classical Sturm-Liouville operator without
discontinuities.

Sturm-Liouville operators with discontinuities inside the interval arise in
mathematics, mechanics, geophysics, and other fields of science and
technology. The inverse spectral problems of such operators is of central
importance in disciplines ranging from engineering to the geosciences. For
example, discontinuous inverse problems appear in electronics for
constructing parameters of heterogeneous electronic lines with desirable
technical characteristics \cite{1,2}. In\ the last decades, inverse spectral
problems for Sturm-Liouville operators with different type discontinuities
have attracted tremendous interest \cite%
{amirov,wei1,hald,mmd,nk,ozkan,ok,shien,wyp,wang,xu,ycf,Yurko,YGMA,YIXUAN,TESCHL}%
. These start with the fundamental work given by V. Ambarzumian \cite{a} and
then by G. Borg \cite{borg}, B. Levitan \cite{levitan1,levitan2}, and V.
Marchenko \cite{marchen,vam} for the classical Sturm-Liouville operators.%

We emphasize that in 1984, O. H. Hald \cite{hald} first generalized
Hochstadt--Lieberman's theorem \cite{lie} to the Sturm--Liouville operator $%
L $, that is, if $H$ is given, $q$ is known on $\left[ \frac{\pi }{2},\pi %
\right] $ and $d\in \left( 0,\frac{\pi }{2}\right) $, then one spectra can
uniquely determine $h,\beta ,\gamma ,d$ and $q$ on $\left[ 0,\pi \right] .$
Motivated by this work, increasing attention has been given to the inverse
spectral problem of recovering the operator $L$ in the self-adjoint case
with partial information given on the potential \cite{shien,xu,ycf}$.$ In
contrast, such inverse spectral problem for the non-self-adjoint case has in
general been studied considerably less, and it is precisely the starting
point of this paper.$\ $We investigate the uniqueness problem of determining
the non-self-adjoint operator $L$ with only partial information of $q,$ of
the eigenvalues, and of the generalized norming constants. What should be
noted is that in the non-self-adjoint setting, complex eigenvalues and
multiple eigenvalues may appear, and thus many new ideas and additional
effort are required. Before describing the content of this paper, let us
first give some notations and basic facts.

To avoid too many case distinctions in the proofs of this paper, we assume
that $h\in
\mathbb{C}
$. Nevertheless, we expect that the method of the paper can be applied in
the case $h=\infty $. For simplicity we use the notations $B$ and $B^{\infty
}$ for the boundary value problems corresponding to $L$ with $H\in
\mathbb{C}
\ $and $H=\infty ,$ respectively. Assume that $\varphi \left( x,\lambda
\right) $, $\psi \left( x,\lambda \right) ,$ $\psi ^{\infty }\left(
x,\lambda \right) $ are solutions of the equation
\begin{equation}
\ell y=-y^{\prime \prime }+q\left( x\right) y=\lambda y  \label{ly1}
\end{equation}%
satisfying the discontinuous conditions (\ref{DCs}) and the initial
conditions
\begin{eqnarray*}
\varphi \left( 0,\lambda \right) &=&1,\left. \frac{d\varphi \left( x,\lambda
\right) }{dx}\right\vert _{x=0}=h\in
\mathbb{C}
, \\
\psi \left( \pi ,\lambda \right) &=&1,\left. \frac{d\psi \left( x,\lambda
\right) }{dx}\right\vert _{x=\pi }=-H\in
\mathbb{C}
, \\
\psi ^{\infty }\left( \pi ,\lambda \right) &=&0,\text{ }\left. \frac{d\psi
^{\infty }\left( x,\lambda \right) }{dx}\right\vert _{x=\pi }=1,
\end{eqnarray*}%
respectively. Then it is easy to see that eigenvalues of $B$ and $B^{\infty
} $ are precisely the zeros of
\begin{equation}
\Delta \left( \lambda \right) :=\left\langle \psi \left( x,\lambda \right)
,\varphi \left( x,\lambda \right) \right\rangle =V\left( \varphi \right)
=-U\left( \psi \right)  \label{w}
\end{equation}%
and%
\begin{equation}
\Delta ^{\infty }\left( \lambda \right) :=\left\langle \psi ^{\infty }\left(
x,\lambda \right) ,\varphi \left( x,\lambda \right) \right\rangle
=-V^{\infty }\left( \varphi \right) =-U\left( \psi ^{\infty }\right) ,
\label{w2}
\end{equation}%
respectively, where $\left\langle y(x),z(x)\right\rangle :=y(x)z^{\prime
}(x)-y^{\prime }(x)z(x)$. Thus $\Delta \left( \lambda \right) $ and $\Delta
^{\infty }\left( \lambda \right) $ are called the characteristic functions
of $B$ and $B^{\infty },$ respectively. Throughout this paper, the \textbf{%
algebraic multiplicity} of an eigenvalue is the order of it as a zero of the
corresponding characteristic function. %
\

\begin{notation}
\label{615 copy(1)}$(1)$ We denote by $\sigma \left( B\right) :=\left\{
\lambda _{n}\right\} _{n\in
\mathbb{N}
_{0}}$ the sequence of all the eigenvalues of $B$ and denote by $\sigma
\left( B^{\infty }\right) :=\left\{ \lambda _{n}^{\infty }\right\} _{n\in
\mathbb{N}
_{0}}$ the sequence of all the eigenvalues of $B^{\infty }$. The eigenvalues
are assumed to be repeated according to their algebraic multiplicities and
labeled in order of increasing moduli. In addition, identical eigenvalues
are adjacent.

$(2)$ Denote%
\begin{equation*}
S_{B}:=\left\{ n\in
\mathbb{N}
|\lambda _{n-1}\neq \lambda _{n}\right\} \cup \left\{ 0\right\}
,S_{B^{\infty }}:=\left\{ n\in
\mathbb{N}
|\lambda _{n-1}^{\infty }\neq \lambda _{n}^{\infty }\right\} \cup \left\{
0\right\} .
\end{equation*}

$(3)$ The symbol $m_{n}$ denotes the algebraic multiplicity of the
eigenvalue $\lambda _{n},n\in S_{B},$ and $m_{n}^{\infty }$\ denotes the
algebraic multiplicity of $\lambda _{n}^{\infty },$ $n\in S_{B^{\infty }}.$
For sufficiently large $n$ it is well known that $m_{n}^{\infty }=m_{n}=1$ (see
Lemma\ 2.3 in \cite{YIXUAN}).
\end{notation}

Now we turn to give the definition of the \textit{generalized norming
constants }for the problem $B$. Denote%
\begin{equation}
\kappa _{n+\nu }:=\varphi _{n+\nu }\left( \pi \right) ,\text{ }\alpha
_{n+\nu }:=\int_{0}^{\pi }\psi _{n+\nu }\left( x\right) \psi
_{n+m_{n}-1}\left( x\right) dx\text{,}  \label{ratios}
\end{equation}%
where $n\in S_{B}$, $\nu =0,1,\ldots ,m_{n}-1$, and
\begin{eqnarray}
&&\varphi _{n+\nu }\left( x\right):=\varphi _{\nu }\left( x,\lambda
_{n}\right) :=\dfrac{1}{\nu !}\left. \dfrac{d^{\nu }}{d\lambda ^{\nu }}%
\varphi \left( x,\lambda \right) \right\vert _{\lambda =\lambda _{n}}\text{,
}  \label{9} \\
&&\psi _{n+\nu }\left( x\right):=\psi _{\nu }\left( x,\lambda _{n}\right) :=%
\dfrac{1}{\nu !}\left. \dfrac{d^{\nu }}{d\lambda ^{\nu }}\psi \left(
x,\lambda \right) \right\vert _{\lambda =\lambda _{n}}.
\end{eqnarray}%
Then $\kappa _{n}$ and $\alpha _{n}$, $n\in
\mathbb{N}
_{0},$ are called the \textit{generalized} \textit{norming constants}
corresponding to $\lambda _{n}.$ To distinguish $\kappa _{n}$ and $\alpha
_{n},$ in this paper, $\kappa _{n}$ is called the \textit{generalized}
\textit{ratio}, and $\alpha _{n}$ is called the \textit{generalized
normalizing constant}. Moreover, it follows from \cite[Theorem 4.1]{YIXUAN}
that for $n\in S_{B}$, $\nu =0,1,\ldots ,m_{n}-1$,

\begin{equation}
\left. \frac{d^{m_{n}+\nu }\Delta \left( \lambda \right) }{d\lambda
^{m_{n}+\nu }}\right\vert _{\lambda =\lambda _{n}}=-\left( m_{n}+\nu \right)
!\sum_{j=0}^{\nu }\kappa _{n+j}\alpha _{n+\nu -j}\text{.}  \label{24}
\end{equation}%
Note that when the multiplicity $m_{n}=1$, the generalized \textit{norming
constants} $\kappa _{n}$ and $\alpha _{n}$ coincide with the \textit{norming
constants} for the operator $L$ in the self-adjoint case (see \cite{Yurko}).

Actually, $\varphi _{\nu }\left( x,\lambda _{n}\right) $ and $\psi _{\nu
}\left( x,\lambda _{n}\right) $ are the generalized eigenfunctions of $B$
corresponding to the eigenvalue $\lambda _{n},$ $n\in S_{B}.$ In fact, for $%
\nu =1,2,\ldots ,m_{n}-1,$ we notice that%
\begin{equation}
\left\{
\begin{array}{l}
\ell \varphi _{\nu }\left( x,\lambda _{n}\right) =\lambda _{n}\varphi _{\nu
}\left( x,\lambda _{n}\right) +\varphi _{\nu -1}\left( x,\lambda _{n}\right)
\text{, } \\
\varphi _{\nu }\left( d+0,\lambda _{n}\right) =\beta \varphi _{\nu }\left(
d-0,\lambda _{n}\right) \text{, } \\
\varphi _{\nu }^{\prime }\left( d+0,\lambda _{n}\right) =\beta ^{-1}\varphi
_{\nu }^{\prime }\left( d-0,\lambda _{n}\right) +\gamma \varphi _{\nu
}\left( d-0,\lambda _{n}\right) \text{, } \\
\varphi _{\nu }\left( 0,\lambda _{n}\right) =\varphi _{\nu }^{\prime }\left(
0,\lambda _{n}\right) =0\text{,}%
\end{array}%
\right.  \label{7}
\end{equation}%
\begin{equation}
\left\{
\begin{array}{l}
\ell \psi _{\nu }\left( x,\lambda _{n}\right) =\lambda _{n}\psi _{\nu
}\left( x,\lambda _{n}\right) +\psi _{\nu -1}\left( x,\lambda _{n}\right)
\text{,} \\
\psi _{\nu }\left( d+0,\lambda _{n}\right) =\beta \psi _{\nu }\left(
d-0,\lambda _{n}\right) \text{, } \\
\psi _{\nu }^{\prime }\left( d+0,\lambda _{n}\right) =\beta ^{-1}\psi _{\nu
}^{\prime }\left( d-0,\lambda _{n}\right) +\gamma \psi _{\nu }\left(
d-0,\lambda _{n}\right) \text{, } \\
\psi _{\nu }\left( \pi ,\lambda _{n}\right) =\psi _{\nu }^{\prime }\left(
\pi ,\lambda _{n}\right) =0\text{.}%
\end{array}%
\right.  \label{8}
\end{equation}%
and
\begin{eqnarray}
\frac{1}{\nu !}\Delta ^{\left( \nu \right) }\left( \lambda _{n}\right)
&=&\varphi _{\nu }^{\prime }\left( \pi ,\lambda _{n}\right) +H\varphi _{\nu
}\left( \pi ,\lambda _{n}\right) =0,\text{ }  \label{3f} \\
\frac{1}{\nu !}\Delta ^{\left( \nu \right) }\left( \lambda _{n}\right)
&=&-\psi _{\nu }^{\prime }\left( 0,\lambda _{n}\right) +h\psi _{\nu }\left(
0,\lambda _{n}\right) =0\text{.}
\end{eqnarray}

\begin{remark}
\label{615}Now we define the generalized norming constants for the problem $%
B^{\infty },$
\begin{equation}
\kappa _{n+\nu }^{\infty }:=\left. \frac{d\varphi _{n+\nu }^{\infty }\left(
x\right) }{dx}\right\vert _{x=\pi },\text{ }\alpha _{n+\nu }^{\infty
}:=\int_{0}^{\pi }\psi _{n+\nu }^{\infty }\left( x\right) \psi
_{n+m_{n}^{\infty }-1}^{\infty }\left( x\right) dx\text{,}  \label{64}
\end{equation}%
where $n\in S_{B^{\infty }}$, $\nu =0,1,\ldots ,m_{n}^{\infty }-1,$ and
\begin{eqnarray}
&&\varphi _{n+\nu }^{\infty }\left( x\right):=\varphi _{\nu }\left(
x,\lambda _{n}^{\infty }\right) :=\dfrac{1}{\nu !}\left. \dfrac{d^{\nu }}{%
d\lambda ^{\nu }}\varphi \left( x,\lambda \right) \right\vert _{\lambda
=\lambda _{n}^{\infty }},\text{ }  \label{61} \\
&&\psi _{n+\nu }^{\infty }\left( x\right):=\psi _{\nu }^{\infty }\left(
x,\lambda _{n}^{\infty }\right) :=\dfrac{1}{\nu !}\left. \dfrac{d^{\nu }}{%
d\lambda ^{\nu }}\psi ^{\infty }\left( x,\lambda \right) \right\vert
_{\lambda =\lambda _{n}^{\infty }}.
\end{eqnarray}%
Then one can also deduce that for $n\in S_{B^{\infty }}$, $\nu =0,1,\ldots
,m_{n}^{\infty }-1,$
\begin{equation}
\left. \frac{d^{m_{n}^{\infty }+\nu }\Delta ^{\infty }\left( \lambda \right)
}{d\lambda ^{m_{n}^{\infty }+\nu }}\right\vert _{\lambda =\lambda
_{n}^{\infty }}=-\left( m_{n}^{\infty }+\nu \right) !\sum_{j=0}^{\nu }\kappa
_{n+j}^{\infty }\alpha _{n+\nu -j}^{\infty }\ .  \label{63}
\end{equation}
\end{remark}

In \cite{YIXUAN}, Y. Liu, G. Shi and J. Yan studied the uniqueness spectral
problem of recovering the non-self-adjoint operator $L\ $from one of the
following spectral characteristics: (1) $\Gamma _{1}:=\left\{ \lambda
_{n},\alpha _{n}\right\} _{n\in
\mathbb{N}
_{0}};$ (2) $\Gamma _{2}:=\left\{ \lambda _{n},\lambda _{n}^{\infty
}\right\} _{n\in
\mathbb{N}
_{0}};$ (3) the Weyl function $M\left( \lambda \right) :=\frac{\Delta
^{\infty }\left( \lambda \right) }{\Delta \left( \lambda \right) }.$
This motivates us to investigate the inverse spectral problem with partial
information given on the potential. More precisely, assume that $q$ is known
on $\left[ b,\pi \right] $ for some constant $b\in \left( 0,\pi \right] ,$
then the uniqueness theorems of this paper will be given in three cases: $%
b\in \left( d,\pi \right] ,$ $b=d,$ $b\in \left( 0,d\right) ,$ where $d$ is
the discontinuous point. In the case of $b\in \left( d,\pi \right] $ or $b=d$%
, we show that $h,$ $\beta ,$ $\gamma \ $and $q$ on $\left[ 0,\pi \right] \ $%
can be uniquely determined by partial information of the eigenvalues $%
\lambda _{n},$ $\lambda _{n}^{\infty }$, and of the generalized normalizing
constants $\alpha _{n},$ $\alpha _{n}^{\infty }$; the uniqueness problem is
also considered under the same circumstances but with the normalizing
constants $\alpha _{n},$ $\alpha _{n}^{\infty }$ replaced by ratios $\kappa
_{n},$ $\kappa _{n}^{\infty }.$ Moreover, for the case $b\in \left(
0,d\right) ,$ similar uniqueness results can be established with the
additional condition that $\beta ,$ $\gamma \ $are known a priori.

We mention that in 1999, F. Gesztesy and B. Simon \cite{ges2} considered the
classical self-adjoint Sturm-Liouville operators and presented a
generalization of Hochstadt--Lieberman theorem to the case where the
potential $q$ is known on a larger interval $\left[ a,\pi \right] $ with $%
a\in \left( 0,\frac{\pi }{2}\right] $ and the set of common eigenvalues is
sufficiently large. Later, G. Wei, H. K. Xu and Z. Wei \cite{wei,zhaoying}
provided some uniqueness results for classical self-adjoint Sturm-Liouville
operators with only partial information on $q,$ on the eigenvalues, and on
the norming constants. While our results are generalizations of the
uniqueness theorems established in \cite{ges2,wei,zhaoying},
the non-self-adjointness and the presence of discontinuities produce
essential qualitative modifications in the investigation of the operator $L$%
. To the best of our knowledge, the uniqueness theorems obtained in this
paper have not yet been developed even for the non-self-adjoint classical
Sturm-Liouville operators (i.e., the case of $\beta =1,$ $\gamma =0$) and
the self-adjoint Sturm-Liouville operators with discontinuous conditions
inside (i.e., the real-valued case).

In addition, we show that less knowledge of eigenvalues and norming
constants can be required if the potential $q$ satisfies a local smoothness
condition, which is a generalization of the results in \cite%
{ges2,wei,zhaoying}. We notice that the key technique in \cite%
{ges2,wei,zhaoying} relies on the high-energy asymptotic expansion of the
Weyl $m$-function \cite{weyl}, however, in our non-self-adjoint situation,
an entirely different approach, based on the asymptotic expansion of the
fundamental solutions of the equation $\left( \ref{ly1}\right) ,$ is
developed (see Proposition \ref{ooo copy(1)}). Now we briefly present some
of these uniqueness results (Theorem \ref{theorem}, Theorem \ref{theorem
copy(4)}, Remark $\ref{theorem copy(6)},$ Corollary \ref{corollary copy(1)}--%
\ref{corollary copy(3)}) as follows.

(S1) We prove that if $q\ $is assumed to be $C^{m}$ near $\pi ,\ $then $h,$ $%
\beta ,$ $\gamma $ and $q$ on $\left[ 0,\pi \right] $ can be uniquely
determined by the values of $q^{\left( j\right) }\left( \pi \right) ,$ $%
j=0,1,\ldots ,m,$ $\left\{ \lambda _{n}\right\} _{n\in
\mathbb{N}
_{0}\backslash \Lambda _{1}}$ $($a subsequence of $\sigma \left( B\right) ),$
and $\left\{ \lambda _{n}^{\infty }\right\} _{n\in
\mathbb{N}
_{0}\backslash \Lambda _{1}^{\infty }}$ $($a subsequence of $\sigma \left(
B^{\infty }\right) ),$ where $\#\Lambda _{1}+\#\Lambda _{1}^{\infty }=$ $%
\left[ \frac{m+2}{2}\right] .$

(S2) When $d\in \left( 0,\frac{\pi }{2}\right) ,\ $we prove that if $q$ is $%
C^{m}$ near $\frac{\pi }{2}\ $and $q\ $on $\left[ \frac{\pi }{2},\pi \right]
\ $is known a priori, then $h,$ $\beta ,$ $\gamma $ and $q$ on $\left[ 0,\pi %
\right] $ can be uniquely determined by all the eigenvalues $\left\{ \lambda
_{n}\right\} _{n\in
\mathbb{N}
_{0}}$ of $B$ except for $\left( \left[ \frac{m+2}{2}\right] \right) ,$ or
all the eigenvalues $\left\{ \lambda _{n}^{\infty }\right\} _{n\in
\mathbb{N}
_{0}}$ of $B^{\infty }$ except for $\left( \left[ \frac{m+1}{2}\right]
\right) ;$ when $d=\frac{\pi }{2},$ the same statement holds if $\beta ,$ $%
\gamma \ $are additionally assumed to be known a priori$.$

Here is a sketch of the contents of this paper. In Section 2, we provide
some preliminary lemmas which will be used to prove the main results. In
Section 3, assume that $q$ is known on $\left[ b,\pi \right] $ for some
constant $b\in \left( 0,\pi \right] ,$ then we discuss the uniqueness
theorems for three cases: $b\in \left( d,\pi \right] $, $b=d$, and $b\in
\left( 0,d\right) .$ Finally, the appendix is devoted to present an
important proposition $\left( \text{see Proposition \ref{ooo copy(1)}}%
\right) ,$ which is necessary to prove our principal results.

We conclude this introduction by briefly summarizing some of the notations
used in this paper.

\begin{notation}
\label{615 copy(2)}$%
\mathbb{C}
$ denotes the complex plane$.$ $%
\mathbb{N}
$ denotes the set of positive integers$\ $and $%
\mathbb{N}
_{0}$ denotes the set of nonnegative integers$.$ Given a set $A,$ the symbol
$\#A$ will be used to denote the number of elements in $A.$ Moreover, given
a sequence $X:=\{x_{n}\}_{n=0}^{\infty }$ of complex numbers, we use the
notation $X_{1}<<X\ $to denote that $X_{1}$ is a subsequence of $X,$ and in
addition, $\widehat{X}$ $:=\underset{n\in
\mathbb{N}
_{0}}{\cup }\{x_{n}\},$ $N_{X}\left( t\right) :=\#\{n\in
\mathbb{N}
_{0}:\left\vert x_{n}\right\vert <t\}\ $for each $t\geq 0.$
\end{notation}

\section{Preliminaries}

In this section, we provide some preliminaries which will be used in Section
3 to prove the main results.

In order to prove the uniqueness theorems, together with $B$ $\left(
B^{\infty }\right) ,$ we consider the boundary value problem $\widetilde{B}$
$\left( \widetilde{B}^{\infty }\right) $ of the same form but with different
coefficients $\tilde{q},$ $\tilde{h},$ $\widetilde{H},$ $\widetilde{\beta },$
$\widetilde{\gamma }$ and $\widetilde{d}.$ We agree that if a certain symbol
$\xi $ denotes an object related to $B$ or $B^{\infty }$, then $\tilde{\xi}$
will denote the analogous object related to $\tilde{B}$ or $\widetilde{B}%
^{\infty }$, and $\hat{\xi}:=\xi -\tilde{\xi}$.

Now we introduce an entire function of $\lambda \in
\mathbb{C}
,$
\begin{equation}
F\left( \lambda \right) :=\left. \left\langle \varphi \left( x,\lambda
\right) ,\widetilde{\varphi }\left( x,\lambda \right) \right\rangle
\right\vert _{x=\pi }.  \label{FFF}
\end{equation}%
From \textrm{\cite[Theorem 5.2 and Remark 1]{YIXUAN}, }the following result
can be given.

\begin{lemma}
\label{unique by F}Suppose that $F\left( \lambda \right) \equiv 0$, then $q$
$=\tilde{q}$ a.e. on $\left[ 0,\pi \right] ,$ $h=\tilde{h},$ $\beta =%
\widetilde{\beta },$ $\gamma =\widetilde{\gamma },$ $d=\widetilde{d}.$
\end{lemma}

It should be noted that our main results are based on Lemma \ref{unique by F}%
. Next, we give an important lemma, which plays a key role in this paper.

\begin{lemma}
\label{F}Suppose that $H=\widetilde{H}\in
\mathbb{C}
\cup \left\{ \infty \right\} .$ If $\lambda _{n}=\widetilde{\lambda }_{%
\widetilde{n}}$ for some $n\in S_{B},$ $\widetilde{n}\in S_{\widetilde{B}},$
and $m_{n}=\widetilde{m}_{\widetilde{n}},$ then
\begin{equation}
\left. \frac{d^{k}}{d\lambda ^{k}}F\left( \lambda \right) \right\vert
_{\lambda =\lambda _{n}}=0\ \text{for }k=0,1,\ldots ,m_{n}-1;
\label{d equal}
\end{equation}%
In addition, if $\alpha _{n+\nu }=\widetilde{\alpha }_{\widetilde{n}+\nu },$
$\nu =0,1,\ldots ,k_{n}-1$, where $k_{n}$ is an integer such that $1\leq
k_{n}\leq m_{n},$ then we have
\begin{equation*}
\left. \frac{d^{m_{n}+\nu }}{d\lambda ^{m_{n}+\nu }}F\left( \lambda \right)
\right\vert _{\lambda =\lambda _{n}}=0\ \text{for }\nu =0,1,\ldots ,k_{n}-1,
\end{equation*}%
that is, in this case, the order of $\lambda _{n}$ $($as a zero of $F\left(
\lambda \right) )$ is at least $\left( m_{n}+k_{n}\right) .$ Similar
statement also holds for the case $H=\widetilde{H}=\infty .$
\end{lemma}

\begin{proof}
We first prove the lemma for $H=\widetilde{H}\in
\mathbb{C}
.$ From (\ref{w}) and the definition (\ref{FFF}) of $F\left( \lambda \right)
$, we have%
\begin{equation}
F\left( \lambda \right) =\left\vert
\begin{array}{cc}
\varphi \left( \pi ,\lambda \right) & \widetilde{\varphi }\left( \pi
,\lambda \right) \\
\Delta \left( \lambda \right) & \widetilde{\Delta }\left( \lambda \right)%
\end{array}%
\right\vert .  \label{FLAMUDA}
\end{equation}%
Since $m_{n}=\widetilde{m}_{\widetilde{n}},$\ we know that%
\begin{equation}
\left. \frac{d^{k}}{d\lambda ^{k}}\Delta \left( \lambda \right) \right\vert
_{\lambda =\lambda _{n}}=0,\text{ }\left. \frac{d^{k}}{d\lambda ^{k}}%
\widetilde{\Delta }\left( \lambda \right) \right\vert _{\lambda =\lambda
_{n}}=0\text{ for }k=0,1,\ldots ,m_{n}-1.  \label{dd equal}
\end{equation}%
This directly yields $\left( \ref{d equal}\right) .$ Now we turn to prove
the second part of this lemma. It follows from $\left( \ref{ratios}\right) ,$
$\left( \ref{9}\right) ,$ $\left( \ref{24}\right) $, $\left( \ref{FLAMUDA}%
\right) $ and $\left( \ref{dd equal}\right) $ that for $\nu =0,1,\ldots
,k_{n}-1,$%
\begin{eqnarray}
&&\left. \frac{d^{m_{n}+\nu }}{d\lambda ^{m_{n}+\nu }}F\left( \lambda
\right) \right\vert _{\lambda =\lambda _{n}}  \notag \\
&=&\sum_{j=0}^{m_{n}+\nu }C_{m_{n}+\nu }^{j}\left\vert
\begin{array}{cc}
\frac{d^{m_{n}+\nu -j}\varphi \left( \pi ,\lambda \right) }{d\lambda
^{m_{n}+\nu -j}} & \frac{d^{m_{n}+\nu -j}\widetilde{\varphi }\left( \pi
,\lambda \right) }{d\lambda ^{m_{n}+\nu -j}} \\
\frac{d^{^{j}}\Delta \left( \lambda \right) }{d\lambda ^{^{j}}} & \frac{%
d^{^{j}}\widetilde{\Delta }\left( \lambda \right) }{d\lambda ^{^{j}}}%
\end{array}%
\right\vert _{\lambda =\lambda _{n}}  \notag \\
&=&\sum_{j=m_{n}}^{m_{n}+\nu }C_{m_{n}+\nu }^{j}\left\vert
\begin{array}{cc}
\frac{d^{m_{n}+\nu -j}\varphi \left( \pi ,\lambda \right) }{d\lambda
^{m_{n}+\nu -j}} & \frac{d^{m_{n}+\nu -j}\widetilde{\varphi }\left( \pi
,\lambda \right) }{d\lambda ^{m_{n}+\nu -j}} \\
\frac{d^{^{j}}\Delta \left( \lambda \right) }{d\lambda ^{^{j}}} & \frac{%
d^{^{j}}\widetilde{\Delta }\left( \lambda \right) }{d\lambda ^{^{j}}}%
\end{array}%
\right\vert _{\lambda =\lambda _{n}}  \notag \\
&=&-\sum_{j=m_{n}}^{m_{n}+\nu }\sum\limits_{l=0}^{j-m_{n}}\dfrac{%
C_{m_{n}+\nu }^{j}j!\alpha _{n+j-m_{n}-l}}{l!}\left\vert
\begin{array}{cc}
\frac{d^{m_{n}+\nu -j}\varphi \left( \pi ,\lambda \right) }{d\lambda
^{m_{n}+\nu -j}} & \frac{d^{m_{n}+\nu -j}\widetilde{\varphi }\left( \pi
,\lambda \right) }{d\lambda ^{m_{n}+\nu -j}} \\
\dfrac{d^{l}\varphi \left( \pi ,\lambda \right) }{d\lambda ^{l}} & \dfrac{%
d^{l}\widetilde{\varphi }\left( \pi ,\lambda \right) }{d\lambda ^{l}}%
\end{array}%
\right\vert _{\lambda =\lambda _{n}}  \label{posinega} \\
&=&-\sum\limits_{l=0}^{\nu }\sum_{j=m_{n}+l}^{m_{n}+\nu }\dfrac{C_{m_{n}+\nu
}^{j}j!\alpha _{n+j-m_{n}-l}}{l!}\left\vert
\begin{array}{cc}
\frac{d^{m_{n}+\nu -j}\varphi \left( \pi ,\lambda \right) }{d\lambda
^{m_{n}+\nu -j}} & \frac{d^{m_{n}+\nu -j}\widetilde{\varphi }\left( \pi
,\lambda \right) }{d\lambda ^{m_{n}+\nu -j}} \\
\dfrac{d^{l}\varphi \left( \pi ,\lambda \right) }{d\lambda ^{l}} & \dfrac{%
d^{l}\widetilde{\varphi }\left( \pi ,\lambda \right) }{d\lambda ^{l}}%
\end{array}%
\right\vert _{\lambda =\lambda _{n}}.  \notag
\end{eqnarray}%
Let $\widehat{l}=m_{n}+\nu -l,$ $\widehat{j}=m_{n}+\nu -j.$ Then
\begin{eqnarray*}
&&\left. \frac{d^{m_{n}+\nu }}{d\lambda ^{m_{n}+\nu }}F\left( \lambda
\right) \right\vert _{\lambda =\lambda _{n}} \\
&=&-\sum\limits_{\widehat{l}=m_{n}+\nu }^{m_{n}}\sum_{\widehat{j}=\widehat{l}%
-m_{n}}^{0}\dfrac{C_{m_{n}+\nu }^{\widehat{l}}\widehat{l}!\alpha _{n+%
\widehat{l}-m_{n}-\widehat{j}}}{\widehat{j}!}\left\vert
\begin{array}{cc}
\dfrac{d^{\widehat{j}}\varphi \left( \pi ,\lambda \right) }{d\lambda ^{%
\widehat{j}}} & \dfrac{d^{\widehat{j}}\widetilde{\varphi }\left( \pi
,\lambda \right) }{d\lambda ^{\widehat{j}}} \\
\frac{d^{m_{n}+\nu -\widehat{l}}\varphi \left( \pi ,\lambda \right) }{%
d\lambda ^{m_{n}+\nu -\widehat{l}}} & \frac{d^{m_{n}+\nu -\widehat{l}}%
\widetilde{\varphi }\left( \pi ,\lambda \right) }{d\lambda ^{m_{n}+\nu -%
\widehat{l}}}%
\end{array}%
\right\vert _{\lambda =\lambda _{n}}.
\end{eqnarray*}%
This together $\left( \ref{posinega}\right) $ yield that $\left. \frac{%
d^{m_{n}+\nu }}{d\lambda ^{m_{n}+\nu }}F\left( \lambda \right) \right\vert
_{\lambda =\lambda _{n}}=-\left. \frac{d^{m_{n}+\nu }}{d\lambda ^{m_{n}+\nu }%
}F\left( \lambda \right) \right\vert _{\lambda =\lambda _{n}},$ and hence $%
\left. \frac{d^{m_{n}+\nu }}{d\lambda ^{m_{n}+\nu }}F\left( \lambda \right)
\right\vert _{\lambda =\lambda _{n}}=0\ $for $\nu =0,1,\ldots ,k_{n}-1.$
This proves the lemma for the case $H=\widetilde{H}\in
\mathbb{C}
.$ In view of Remark \ref{615} and the fact%
\begin{equation*}
F\left( \lambda \right) :=\left\vert
\begin{array}{cc}
\varphi \left( \pi ,\lambda \right) & \widetilde{\varphi }\left( \pi
,\lambda \right) \\
\varphi ^{\prime }\left( \pi ,\lambda \right) & \widetilde{\varphi }^{\prime
}\left( \pi ,\lambda \right)%
\end{array}%
\right\vert =-\left\vert
\begin{array}{cc}
\Delta ^{\infty }\left( \lambda \right) & \widetilde{\Delta ^{\infty }}%
\left( \lambda \right) \\
\varphi ^{\prime }\left( \pi ,\lambda \right) & \widetilde{\varphi }^{\prime
}\left( \pi ,\lambda \right)%
\end{array}%
\right\vert ,
\end{equation*}%
the lemma for $H=\widetilde{H}=\infty $ can be proved similarly.
\end{proof}

\begin{lemma}
\label{Fqh}Assume that $d=\widetilde{d}\ $and $q$ $=\tilde{q}$ a.e. on $%
\left[ b,\pi \right] $ for some $b\in \left( 0,\pi \right] ,$ then following
expressions hold$:$
\begin{eqnarray*}
F\left( \lambda \right) &=&\left. \left\langle \varphi \left( x,\lambda
\right) ,\widetilde{\varphi }\left( x,\lambda \right) \right\rangle
\right\vert _{x=b}\text{ for }b\in \left( d,\pi \right] , \\
F\left( \lambda \right) &=&\left. \left\langle \varphi \left( x,\lambda
\right) ,\widetilde{\varphi }\left( x,\lambda \right) \right\rangle
\right\vert _{x=d+0}\text{ for }b=d, \\
F\left( \lambda \right) &=&\left. \left\langle \varphi \left( x,\lambda
\right) ,\widetilde{\varphi }\left( x,\lambda \right) \right\rangle
\right\vert _{x=b}+\left. \left\langle \varphi \left( x,\lambda \right) ,%
\widetilde{\varphi }\left( x,\lambda \right) \right\rangle \right\vert
_{d-0}^{d+0}\text{ for }b\in \left( 0,d\right) .
\end{eqnarray*}
\end{lemma}

\begin{proof}
From the definition $\left( \ref{FFF}\right) $ of $F\left( \lambda \right) ,$
one can easily deduce that%
\begin{eqnarray*}
F\left( \lambda \right) &=&-\int_{0}^{\pi }\hat{q}\left( x\right) \varphi
\left( x,\lambda \right) \widetilde{\varphi }\left( x,\lambda \right)
dx+\left. \left\langle \varphi \left( x,\lambda \right) ,\widetilde{\varphi }%
\left( x,\lambda \right) \right\rangle \right\vert _{x=0} \\
&&+\left. \left\langle \varphi \left( x,\lambda \right) ,\widetilde{\varphi }%
\left( x,\lambda \right) \right\rangle \right\vert _{x=d+0}-\left.
\left\langle \varphi \left( x,\lambda \right) ,\widetilde{\varphi }\left(
x,\lambda \right) \right\rangle \right\vert _{x=d-0}.
\end{eqnarray*}
Hence by $q$ $=\tilde{q}$ a.e. on $\left[ b,\pi \right] $ we infer from the
above equality that
\begin{eqnarray*}
F\left( \lambda \right) &=&-\int_{0}^{b}\hat{q}\left( x\right) \varphi
\left( x,\lambda \right) \widetilde{\varphi }\left( x,\lambda \right)
dx+\left. \left\langle \varphi \left( x,\lambda \right) ,\widetilde{\varphi }%
\left( x,\lambda \right) \right\rangle \right\vert _{x=0} \\
&&+\left. \left\langle \varphi \left( x,\lambda \right) ,\widetilde{\varphi }%
\left( x,\lambda \right) \right\rangle \right\vert _{x=d+0}-\left.
\left\langle \varphi \left( x,\lambda \right) ,\widetilde{\varphi }\left(
x,\lambda \right) \right\rangle \right\vert _{x=d-0}.
\end{eqnarray*}%
Therefore, this lemma can be directly proved by the following facts%
\begin{eqnarray*}
&&-\int_{0}^{b}\hat{q}\left( x\right) \varphi \left( x,\lambda \right)
\widetilde{\varphi }\left( x,\lambda \right) dx \\
&=&\left\{
\begin{array}{l}
\left. \left\langle \varphi \left( x,\lambda \right) ,\widetilde{\varphi }%
\left( x,\lambda \right) \right\rangle \right\vert _{d+0}^{b}+\left.
\left\langle \varphi \left( x,\lambda \right) ,\widetilde{\varphi }\left(
x,\lambda \right) \right\rangle \right\vert _{0}^{d-0}\text{ for }b\in
\left( d,\pi \right] , \\
\left. \left\langle \varphi \left( x,\lambda \right) ,\widetilde{\varphi }%
\left( x,\lambda \right) \right\rangle \right\vert _{0}^{d-0}\text{ for }b=d,
\\
\left. \left\langle \varphi \left( x,\lambda \right) ,\widetilde{\varphi }%
\left( x,\lambda \right) \right\rangle \right\vert _{0}^{b}\text{ for }b\in
\left( 0,d\right) .%
\end{array}%
\right.
\end{eqnarray*}
\end{proof}

\begin{lemma}
\label{jjgj}{As $\left\vert \lambda \right\vert \rightarrow \infty $,
\begin{equation}
\varphi \left( x,\lambda \right) =\left\{
\begin{array}{l}
\cos \left( \sqrt{\lambda }x\right) +O\left( \frac{\exp \left( \left\vert
\mathrm{Im}\sqrt{\lambda }\right\vert x\right) }{\sqrt{\lambda }}\right)
,x<d, \\
\left( b_{1}\cos \left( \sqrt{\lambda }x\right) +b_{2}\cos \left( \sqrt{%
\lambda }\left( 2d-x\right) \right) \right) +O\left( \frac{\exp \left(
\left\vert \mathrm{Im}\sqrt{\lambda }\right\vert x\right) }{\sqrt{\lambda }}%
\right) ,x>d,%
\end{array}%
\right.  \label{1}
\end{equation}%
\begin{equation}
\varphi ^{\prime }\left( x,\lambda \right) =\left\{
\begin{array}{l}
-\sqrt{\lambda }\sin \left( \sqrt{\lambda }x\right) +O\left( \exp \left(
\left\vert \mathrm{Im}\sqrt{\lambda }\right\vert x\right) \right) ,x<d, \\
\sqrt{\lambda }\left( -b_{1}\sin \left( \sqrt{\lambda }x\right) +b_{2}\sin
\left( \sqrt{\lambda }\left( 2d-x\right) \right) \right) +O\left( \exp
\left( \left\vert \mathrm{Im}\sqrt{\lambda }\right\vert x\right) \right)
,x>d,%
\end{array}%
\right.  \label{11}
\end{equation}%
where $b_{1}=\dfrac{\beta +\beta ^{-1}}{2}$ and $b_{2}=\dfrac{\beta -\beta
^{-1}}{2}$}.

\begin{proof}
See \cite[p.145-146]{Yurko}.
\end{proof}
\end{lemma}

\begin{remark}
If $\lambda =iy$ with $y\in
\mathbb{R}
,$ then by Lemma \ref{jjgj}, $\left( \ref{w}\right) $ and $\left( \ref{w2}%
\right) ,$ one deduces that as $\left\vert y\right\vert \rightarrow \infty ,$
\begin{eqnarray}
\left\vert \Delta \left( iy\right) \right\vert &=&\frac{b_{1}}{2}\left\vert
y\right\vert ^{\frac{1}{2}}\exp \left( \left\vert \mathrm{Im}\sqrt{iy}%
\right\vert \pi \right) \left( 1+o\left( 1\right) \right) ,  \label{delta1}
\\
\left\vert \Delta ^{\infty }\left( iy\right) \right\vert &=&\frac{b_{1}}{2}%
\exp \left( \left\vert \mathrm{Im}\sqrt{iy}\right\vert \pi \right) \left(
1+o\left( 1\right) \right) ,  \label{delta2} \\
\left\vert \varphi \left( b,iy\right) \right\vert &=&\left\{
\begin{array}{l}
\frac{b_{1}}{2}\exp \left( \left\vert \mathrm{Im}\sqrt{iy}\right\vert
b\right) \left( 1+o\left( 1\right) \right) \ \text{for }b>d, \\
\frac{1}{2}\exp \left( \left\vert \mathrm{Im}\sqrt{iy}\right\vert b\right)
\left( 1+o\left( 1\right) \right) \ \text{for }b<d,%
\end{array}%
\right.  \label{faib} \\
\left\vert \varphi \left( d+0,iy\right) \right\vert &=&\frac{\beta }{2}\exp
\left( \left\vert \mathrm{Im}\sqrt{iy}\right\vert d\right) \left( 1+o\left(
1\right) \right) .  \label{faid}
\end{eqnarray}
\end{remark}


We conclude this section with two lemmas $\left( \text{see Lemma \ref{number}
and Lemma \ref{proposition}}\right) $, which will be used in Section 3 to
prove our main results. Now we first give some notations and basic facts.%
%

Recall that $\sigma \left( B\right) :=\left\{ \lambda _{n}\right\} _{n\in
\mathbb{N}
_{0}}$ and $\sigma \left( B^{\infty }\right) :=\left\{ \lambda _{n}^{\infty
}\right\} _{n\in
\mathbb{N}
_{0}}$ are the sequences consisting of all the eigenvalues of $B$ and $%
B^{\infty },$ respectively. By the asymptotics of the eigenvalues $\lambda
_{n}$ and $\lambda _{n}^{\infty }$ \cite{YIXUAN}, it is easy to see that
there exist constants $r_{1}$ and $r_{2}$ such that%
\begin{equation*}
\min\limits_{n\in
\mathbb{N}
_{0}}\left\{ \text{Re}\lambda _{n}\right\} \geq r_{1},\text{ }%
\min\limits_{n\in
\mathbb{N}
_{0}}\left\{ \text{Re}\lambda _{n}^{\infty }\right\} \geq r_{2}.\text{ }
\end{equation*}%
Hence by adding (if necessary) a sufficiently large constant to the
potential coefficient $q$, throughout this paper we may assume that
\begin{equation}
N_{\sigma \left( B\right) }\left( t\right) =N_{\sigma \left( B^{\infty
}\right) }\left( t\right) =0\text{ for }t\leq 1.  \label{0000}
\end{equation}%
By Lemma \ref{jjgj} one can easily deduce that $\Delta \left( \lambda
\right) $ and $\Delta ^{\infty }\left( \lambda \right) $ are entire in $%
\lambda \in
\mathbb{C}
$ of order $\frac{1}{2},$ and hence by Hadamard's Factorization Theorem \cite%
[Ch. I]{levision}, there exist constants $C_{B}$ and $C_{B^{\infty }}$ such
that%
\begin{eqnarray}
\Delta \left( \lambda \right) &=&C_{B}\prod\limits_{n=0}^{\infty }\left( 1-%
\frac{\lambda }{\lambda _{n}}\right) ,  \label{ch1} \\
\Delta ^{\infty }\left( \lambda \right) &=&C_{B^{\infty
}}\prod\limits_{n=0}^{\infty }\left( 1-\frac{\lambda }{\lambda _{n}^{\infty }%
}\right) .  \label{ch2}
\end{eqnarray}%
Moreover, it follows from \cite[Ch. I, Theorem 4]{levision} that%
\begin{equation}
N_{\sigma \left( B\right) }\left( t\right) \leq C\left\vert t\right\vert
^{\rho }\text{ and }N_{\sigma \left( B^{\infty }\right) }\left( t\right)
\leq C\left\vert t\right\vert ^{\rho }\text{for all }\rho >\frac{1}{2},
\label{rrrrr}
\end{equation}%
where $C$ is some positive constant.

\begin{lemma}
\label{number}Let $X:=\{x_{n}\}_{n=0}^{\infty }$ with $0<\left\vert
x_{0}\right\vert \leq \left\vert x_{1}\right\vert \leq \left\vert
x_{2}\right\vert \leq \cdots $ be a sequence satisfying
\begin{equation}
\max\limits_{n\in
\mathbb{N}
_{0}}\left\vert \text{Im}x_{n}\right\vert \leq c_{1\text{ \ }}\text{for some
}c_{1}>0,  \label{c111}
\end{equation}%
and%
\begin{eqnarray}
N_{X}\left( t\right) &=&0\text{ for }t\leq 1,  \label{000} \\
N_{X}\left( t\right) &\leq &C\left\vert t\right\vert ^{\rho }\text{ for all }%
\rho >\rho _{0},  \label{rrr}
\end{eqnarray}%
where $C$ is some positive constant and $\rho _{0}\in \left( 0,1\right) $ is
fixed$.$ If there exist real constants $l_{1},$ $l_{2},$ $l_{3}$ such that
for sufficiently large $t\in
\mathbb{R}
,$%
\begin{equation}
N_{X}\left( t\right) \geq l_{1}N_{\sigma \left( B\right) }\left( t\right)
+l_{2}N_{\sigma \left( B^{\infty }\right) }\left( t\right) +l_{3},
\label{hop}
\end{equation}%
then there exists a constant $M>0$ such that for sufficiently large $%
\left\vert y\right\vert $ $(y$ being real$)$
\begin{equation*}
\left\vert G_{X}\left( iy\right) \right\vert \geq M\left\vert y\right\vert ^{%
\frac{l_{1}}{2}+l_{3}}e^{\pi \left( l_{1}+l_{2}\right) \left\vert \text{Im}%
\sqrt{iy}\right\vert },
\end{equation*}%
where $G_{X}\left( \lambda \right) :=\prod\limits_{n=0}^{\infty }\left( 1-%
\frac{\lambda }{x_{n}}\right) .$
\end{lemma}

\begin{proof}
Note that%
\begin{equation}
\frac{d}{dt}\left[ \frac{1}{2}\ln \left( 1+\frac{y^{2}}{t^{2}}\right) \right]
=-\frac{y^{2}}{t^{3}+ty^{2}}.  \label{d}
\end{equation}%
Then by $\left( \ref{000}\right) ,$ $\left( \ref{rrr}\right) ,$ $\left( \ref%
{d}\right) $ and integration by parts, we infer that for $y\in
\mathbb{R}
$,
\begin{eqnarray*}
&&\ln \left\vert G_{X}\left( iy\right) \right\vert \\
&=&\frac{1}{2}\sum_{n=0}^{\infty }\ln \frac{\left[ \left( 1-\frac{iy}{x_{n}}%
\right) \left( 1+\frac{iy}{\overline{x_{n}}}\right) \right] }{1+\frac{y^{2}}{%
\left\vert x_{n}\right\vert ^{2}}}+\frac{1}{2}\sum_{n=0}^{\infty }\ln \left(
1+\frac{\left\vert y\right\vert ^{2}}{\left\vert x_{n}\right\vert ^{2}}%
\right) \\
&=&\frac{1}{2}\sum_{n=0}^{\infty }\ln \frac{\left( 1+\frac{2y\text{Im}%
x_{n}+y^{2}}{\left\vert x_{n}\right\vert ^{2}}\right) }{1+\frac{y^{2}}{%
\left\vert x_{n}\right\vert ^{2}}}+\frac{1}{2}\int_{0}^{\infty }\ln \left( 1+%
\frac{y^{2}}{t^{2}}\right) dN_{X}\left( t\right) \\
&=&\frac{1}{2}\sum_{n=0}^{\infty }\ln \frac{\left( 1+\frac{2y\text{Im}%
x_{n}+y^{2}}{\left\vert x_{n}\right\vert ^{2}}\right) }{1+\frac{y^{2}}{%
\left\vert x_{n}\right\vert ^{2}}}+\int_{1}^{\infty }\frac{y^{2}}{%
t^{3}+ty^{2}}N_{X}\left( t\right) dt.
\end{eqnarray*}%
Similarly, by $\left( \ref{0000}\right) ,$ $\left( \ref{rrrrr}\right) \ $and
$\left( \ref{d}\right) $ we deduce that
\begin{eqnarray*}
l_{1}\ln \left\vert G_{\sigma \left( B\right) }\left( iy\right) \right\vert
&=&\frac{l_{1}}{2}\sum_{n=0}^{\infty }\ln \frac{\left( 1+\frac{2y\text{Im}%
\lambda _{n}+y^{2}}{\left\vert \lambda _{n}\right\vert ^{2}}\right) }{1+%
\frac{y^{2}}{\left\vert \lambda _{n}\right\vert ^{2}}}+l_{1}\int_{1}^{\infty
}\frac{y^{2}}{t^{3}+ty^{2}}N_{\sigma \left( B\right) }\left( t\right) dt, \\
l_{2}\ln \left\vert G_{\sigma \left( B^{\infty }\right) }\left( iy\right)
\right\vert &=&\frac{l_{2}}{2}\sum_{n=0}^{\infty }\ln \frac{\left( 1+\frac{2y%
\text{Im}\lambda _{n}^{\infty }+y^{2}}{\left\vert \lambda _{n}^{\infty
}\right\vert ^{2}}\right) }{1+\frac{y^{2}}{\left\vert \lambda _{n}^{\infty
}\right\vert ^{2}}}+l_{2}\int_{1}^{\infty }\frac{y^{2}}{t^{3}+ty^{2}}%
N_{\sigma \left( B^{\infty }\right) }\left( t\right) dt.
\end{eqnarray*}%
where%
\begin{equation}
G_{\sigma \left( B\right) }\left( \lambda \right)
:=\prod\limits_{n=0}^{\infty }\left( 1-\frac{\lambda }{\lambda _{n}}\right)
\text{ and }G_{\sigma \left( B^{\infty }\right) }\left( \lambda \right)
:=\prod\limits_{n=0}^{\infty }\left( 1-\frac{\lambda }{\lambda _{n}^{\infty }%
}\right) .
\end{equation}%
Therefore,
\begin{eqnarray}
&&\ln \left\vert G_{X}\left( iy\right) \right\vert -l_{1}\ln \left\vert
G_{\sigma \left( B\right) }\left( iy\right) \right\vert -l_{2}\ln \left\vert
G_{\sigma \left( B^{\infty }\right) }\left( iy\right) \right\vert  \label{31}
\\
&=&g(y)+\int_{1}^{\infty }\frac{y^{2}}{t^{3}+ty^{2}}\left( N_{X}\left(
t\right) -l_{1}N_{\sigma \left( B\right) }\left( t\right) -l_{2}N_{\sigma
\left( B^{\infty }\right) }\left( t\right) \right) dt,  \notag
\end{eqnarray}%
where
\begin{eqnarray}
g(y):= &&\frac{1}{2}\sum_{n=0}^{\infty }\ln \frac{\left( 1+\frac{2y\text{Im}%
x_{n}+y^{2}}{\left\vert x_{n}\right\vert ^{2}}\right) }{1+\frac{y^{2}}{%
\left\vert x_{n}\right\vert ^{2}}}-\frac{l_{1}}{2}\sum_{n=0}^{\infty }\ln
\frac{\left( 1+\frac{2y\text{Im}\lambda _{n}+y^{2}}{\left\vert \lambda
_{n}\right\vert ^{2}}\right) }{1+\frac{y^{2}}{\left\vert \lambda
_{n}\right\vert ^{2}}}  \label{gy} \\
&&-\frac{l_{2}}{2}\sum_{n=0}^{\infty }\ln \frac{\left( 1+\frac{2y\text{Im}%
\lambda _{n}^{\infty }+y^{2}}{\left\vert \lambda _{n}^{\infty }\right\vert
^{2}}\right) }{1+\frac{y^{2}}{\left\vert \lambda _{n}^{\infty }\right\vert
^{2}}}.  \notag
\end{eqnarray}%
Next, we aim to show that there exists a constant $C_{g}>0$ such that
\begin{equation}
\left\vert g(y)\right\vert \leq C_{g}\text{ for all }y\in
\mathbb{R}
.  \label{34}
\end{equation}%
In fact, we first note that there exist constants $c_{2}$ and $c_{3}$ such
that
\begin{equation}
\max\limits_{n\in
\mathbb{N}
_{0}}\left\vert \text{Im}\lambda _{n}\right\vert \leq c_{2}\text{ and }%
\max\limits_{n\in
\mathbb{N}
_{0}}\left\vert \text{Im}\lambda _{n}^{\infty }\right\vert \leq c_{3},
\label{c1c2}
\end{equation}%
which can be obtained from the asymptotics of the eigenvalues $\lambda _{n}$
and $\lambda _{n}^{\infty }$ \cite{YIXUAN}$.$ In addition,%
\begin{equation}
\frac{d}{dt}\left[ \ln \left( 1+\frac{2\left\vert y\right\vert c_{i}}{%
t^{2}+y^{2}}\right) \right] =-\frac{4\left\vert y\right\vert c_{i}t}{\left(
t^{2}+y^{2}+2\left\vert y\right\vert c_{i}\right) \left( t^{2}+y^{2}\right) }%
,\text{ }i=1,2,3,  \label{dt}
\end{equation}%
where $c_{1}$ is defined by $\left( \ref{c111}\right) $ and $c_{2},$ $c_{3}$
are defined by $\left( \ref{c1c2}\right) .$ Then by $\left( \ref{rrrrr}%
\right) ,$ $\left( \ref{c111}\right) ,$ $\left( \ref{rrr}\right) ,$ $\left( %
\ref{gy}\right) ,$ $\left( \ref{c1c2}\right) ,$ $\left( \ref{dt}\right) $
and integration by parts, we obtain that
\begin{eqnarray*}
\left\vert g(y)\right\vert &\leq &\frac{1}{2}\sum_{n=0}^{\infty }\ln \left(
1+\frac{2\left\vert y\right\vert c_{1}}{\left\vert x_{n}\right\vert
^{2}+\left\vert y\right\vert ^{2}}\right) +\left\vert \frac{l_{1}}{2}%
\right\vert \sum_{n=0}^{\infty }\ln \left( 1+\frac{2\left\vert y\right\vert
c_{2}}{\left\vert \lambda _{n}\right\vert ^{2}+\left\vert y\right\vert ^{2}}%
\right) \\
&&+\left\vert \frac{l_{2}}{2}\right\vert \sum_{n=0}^{\infty }\ln \left( 1+%
\frac{2\left\vert y\right\vert c_{3}}{\left\vert \lambda _{n}^{\infty
}\right\vert ^{2}+\left\vert y\right\vert ^{2}}\right) \\
&=&\frac{1}{2}\int_{1}^{\infty }\ln \left( 1+\frac{2\left\vert y\right\vert
c_{1}}{t^{2}+y^{2}}\right) dN_{X}\left( t\right) +\left\vert \frac{l_{1}}{2}%
\right\vert \int_{1}^{\infty }\ln \left( 1+\frac{2\left\vert y\right\vert
c_{2}}{t^{2}+y^{2}}\right) dN_{\sigma \left( B\right) }\left( t\right) \\
&&+\left\vert \frac{l_{2}}{2}\right\vert \int_{1}^{\infty }\ln \left( 1+%
\frac{2\left\vert y\right\vert c_{3}}{t^{2}+y^{2}}\right) dN_{\sigma \left(
B^{\infty }\right) }\left( t\right) \\
&\leq &2c_{1}\int_{1}^{\infty }N_{X}\left( t\right) \frac{t\left\vert
y\right\vert }{\left( t^{2}+y^{2}\right) ^{2}}dt+2c_{2}\left\vert
l_{1}\right\vert \int_{1}^{\infty }N_{\sigma \left( B\right) }\left(
t\right) \frac{t\left\vert y\right\vert }{\left( t^{2}+y^{2}\right) ^{2}}dt
\\
&&+2c_{3}\left\vert l_{2}\right\vert \int_{1}^{\infty }N_{\sigma \left(
B^{\infty }\right) }\left( t\right) \frac{\left\vert y\right\vert t}{\left(
t^{2}+y^{2}\right) ^{2}}dt \\
&\leq &C_{0}\int_{1}^{\infty }\frac{t^{2}\left\vert y\right\vert }{\left(
t^{2}+y^{2}\right) ^{2}}dt\leq C_{0}\int_{1}^{\infty }\frac{\left\vert
y\right\vert }{t^{2}+y^{2}}dt \\
&=&\frac{C_{0}\pi }{2}-C_{0}\arctan \frac{1}{\left\vert y\right\vert }\text{
}\left( \text{if }y\neq 0\right) ,\text{ }
\end{eqnarray*}%
where $C_{0}$ is some positive constant. This directly yields $\left( \ref%
{34}\right) .$ By hypothesis (\ref{hop}) we know that there exist constants $%
t_{0}\geq 1$ and $C_{1}\geq 0$ such that%
\begin{eqnarray}
N_{X}\left( t\right) -l_{1}N_{\sigma \left( B\right) }\left( t\right)
-l_{2}N_{\sigma \left( B^{\infty }\right) }\left( t\right) &\geq &l_{3},%
\text{ }t\geq t_{0},  \label{32} \\
N_{X}\left( t\right) -l_{1}N_{\sigma \left( B\right) }\left( t\right)
-l_{2}N_{\sigma \left( B^{\infty }\right) }\left( t\right) &\geq &-C_{1},%
\text{ }t\leq t_{0}.  \label{33}
\end{eqnarray}%
Therefore, it follows from (\ref{31}), $\left( \ref{34}\right) ,$ (\ref{32})
and (\ref{33}) that
\begin{eqnarray}
&&\ln \frac{\left\vert G_{X}\left( iy\right) \right\vert }{\left\vert
G_{\sigma \left( B\right) }\left( iy\right) \right\vert ^{l_{1}}\left\vert
G_{\sigma \left( B^{\infty }\right) }\left( iy\right) \right\vert ^{l_{2}}}
\label{2.37} \\
&\geq &-C_{g}-\int_{1}^{t_{0}}\frac{y^{2}}{t^{3}+ty^{2}}C_{1}dt+%
\int_{t_{0}}^{\infty }\frac{y^{2}}{t^{3}+ty^{2}}l_{3}dt  \notag \\
&\geq &-C_{g}-\left( C_{1}+l_{3}\right) \int_{1}^{t_{0}}\frac{y^{2}}{%
t^{3}+ty^{2}}dt+l_{3}\int_{1}^{\infty }\frac{y^{2}}{t^{3}+ty^{2}}dt  \notag
\\
&=&-C_{g}+\left( C_{1}+l_{3}\right) \frac{1}{2}\ln \frac{\left(
t_{0}^{2}+y^{2}\right) }{t_{0}^{2}\left( 1+y^{2}\right) }+\frac{l_{3}}{2}\ln
\left( 1+y^{2}\right) .  \notag
\end{eqnarray}%
In addition, by $\left( \ref{delta1}\right) ,$ $\left( \ref{delta2}\right) ,$
$\left( \ref{ch1}\right) $ and $\left( \ref{ch2}\right) ,$ we infer that $\ $
\begin{equation}
\left\{
\begin{array}{l}
\left\vert G_{\sigma \left( B\right) }\left( iy\right) \right\vert =\frac{%
b_{1}}{2\left\vert C_{B}\right\vert }\left\vert y\right\vert ^{\frac{1}{2}%
}\exp \left( \left\vert \text{Im}\sqrt{iy}\right\vert \pi \right) \left(
1+o\left( 1\right) \right) , \\
\left\vert G_{\sigma \left( B^{\infty }\right) }\left( iy\right) \right\vert
=\frac{b_{1}}{2\left\vert C_{B^{\infty }}\right\vert }\exp \left( \left\vert
\text{Im}\sqrt{iy}\right\vert \pi \right) \left( 1+o\left( 1\right) \right) .%
\text{ }%
\end{array}%
\right.  \label{delte11}
\end{equation}%
Hence it turns out from $\left( \ref{2.37}\right) $ and $\left( \ref{delte11}%
\right) $ that there exists a constant $M>0$ such that%
\begin{equation*}
\left\vert G_{X}\left( iy\right) \right\vert \geq M\left\vert y\right\vert ^{%
\frac{l_{1}}{2}+l_{3}}e^{\pi \left( l_{1}+l_{2}\right) \left\vert \text{Im}%
\sqrt{iy}\right\vert }
\end{equation*}%
for sufficiently large $\left\vert y\right\vert $ and $y\in
\mathbb{R}
.$ This completes the proof.
\end{proof}

\begin{lemma}
\label{proposition}Assume that $g\left( \lambda \right) $ is an entire\
function of order less than one. If $\lim\limits_{\left\vert y\right\vert
\rightarrow \infty ;y\in
\mathbb{R}
}\left\vert g\left( iy\right) \right\vert =0,$ then $g\left( \lambda \right)
\equiv 0.$

\begin{proof}
The proof is referred to \cite{levision,ges2}.
\end{proof}
\end{lemma}

\section{Main Results and Proofs}

Our goal of this section is to give the main results of this paper. Assume
that the potential $q$ is known on $\left[ b,\pi \right] ,$ then due to the
presence of discontinuous conditions at $d\in \left( 0,\pi \right) ,$ the
uniqueness theorems are given for three cases: $b\in \left( d,\pi \right] $,
$b=d,$ and $b\in \left( 0,d\right) .$ In each case, we first study the
uniqueness problem (Theorem \ref{theorem}, Theorem \ref{theorem copy(2)},
Theorem \ref{theorem copy(4)}) when only partial information on $q$, on the
eigenvalues, and on the generalized normalizing constants is available, and
then we investigate the uniqueness problem (Theorem \ref{theorem copy(1)},
Theorem \ref{theorem copy(3)}, Theorem \ref{theorem copy(5)}) under the same
circumstances but with the normalizing constants replaced by ratios.
Unless explicitly stated otherwise, $H\ $and $d$ will be fixed$\ $in this
section. In addition, let us recall Notation \ref{615 copy(1)} and Notation %
\ref{615 copy(2)} given in the introduction.

\subsection{Case I: $q$ is known on $\left[ b,\protect\pi \right] ,$ where $%
b\in \left( d,\protect\pi \right] $}

\subsubsection{Pairs of Eigenvalues and Normalizing Constants}

\begin{Hypothesis}
\label{hypo1}Consider the subsequences $W,$ $W_{1},$ $W^{\infty },$ $%
W_{1}^{\infty }\ $satisfying
\begin{eqnarray*}
&&W_{1}<<W<<\sigma \left( B\right) ,\text{ }W_{1}<<W<<\sigma \left(
\widetilde{B}\right) , \\
&&W_{1}^{\infty }<<W^{\infty }<<\sigma \left( B^{\infty }\right)
,W_{1}^{\infty }<<W^{\infty }<<\sigma \left( \widetilde{B}^{\infty }\right)
\end{eqnarray*}%
and the following conditions:

$(1)$ for any $\lambda _{n}=\widetilde{\lambda }_{\widetilde{n}}\in \widehat{%
W_{1}}\ $where $n\in S_{B}\ $and $\widetilde{n}\in S_{\widetilde{B}},$
suppose that
\begin{equation}
m_{n}=\widetilde{m}_{\widetilde{n}},\text{ }\alpha _{n+\nu }=\widetilde{%
\alpha }_{\widetilde{n}+\nu }\text{ for }\nu =0,1,\ldots ,k_{n}-1,\text{ }
\label{hhypo1}
\end{equation}%
where $k_{n}$ equals the number of occurrences of the eigenvalue $\lambda
_{n}\ $in $W_{1};$

$(2)$ for any $\lambda _{n}^{\infty }=\widetilde{\lambda }_{\widetilde{n}%
}^{\infty }\in \widehat{W_{1}^{\infty }}\ $where $n\in S_{B^{\infty }}$ and $%
\widetilde{n}\in S_{\widetilde{B}^{\infty }},$ suppose that
\begin{equation}
m_{n}^{\infty }=\widetilde{m}_{\widetilde{n}}^{\infty },\text{ }\alpha
_{n+\gamma }^{\infty }=\widetilde{\alpha }_{\widetilde{n}+\gamma }^{\infty
}\ \text{for }\gamma =0,1,\ldots ,k_{n}^{\infty }-1,\text{ }  \label{hhypo2}
\end{equation}%
where $k_{n}^{\infty }$ equals the number of occurrences of the eigenvalue $%
\lambda _{n}^{\infty }\ $in $W_{1}^{\infty }.$
\end{Hypothesis}

\begin{theorem}
\label{theorem}Assume Hypothesis \ref{hypo1} and suppose that $q,$ $%
\widetilde{q}\in $ $C^{m}$ near $b\in $ $\left( d,\pi \right] ,$ $m\in
\mathbb{N}
_{0},$ $q=\widetilde{q}$ a.e. on $\left[ b,\pi \right] $ $($in particular,
for $b=\pi :$ $q^{\left( j\right) }(\pi )=\widetilde{q}^{\left( j\right)
}(\pi )$ for $j=0,1,\ldots ,m),$ and
\begin{eqnarray}
&&N_{W}(t)+N_{W_{1}}(t)+N_{W^{\infty }}(t)+N_{W_{1}^{\infty }}(t)
\label{hypo} \\
&\geq &AN_{\sigma \left( B\right) }(t)+\left( \frac{2b}{\pi }-A\right)
N_{\sigma \left( B^{\infty }\right) }(t)-\frac{A}{2}-\frac{m+1}{2}  \notag
\end{eqnarray}%
for sufficiently large $t\in
\mathbb{R}
.$ Then $h=\widetilde{h},$ $\beta =\widetilde{\beta },$ $\gamma =\widetilde{%
\gamma }$ and $q=\widetilde{q}$ a.e. on $\left[ 0,\pi \right] .$
\end{theorem}

\begin{remark}
By Remark \ref{ooo copy(4)}, we know that if $q$ and $\widetilde{q}\ $are
assumed to be in $L_{%
\mathbb{C}
}^{1}\left[ 0,\pi \right] ,$ then Theorem \ref{theorem} should be modified
by taking $m=-1.$ Thus for brevity $C^{-1}$ means $L^{1}$ throughout this
paper unless explicitly stated otherwise$.$
\end{remark}

\begin{corollary}
\label{corollary copy(1)}If $q\ $is assumed to be $C^{m}$ near $\pi ,\ $then
$h,$ $\beta ,$ $\gamma $ and $q$ on $\left[ 0,\pi \right] $ can be uniquely
determined by the values of $q^{\left( j\right) }\left( \pi \right) ,$ $%
j=0,1,\ldots ,m,$ $\left\{ \lambda _{n}\right\} _{n\in
\mathbb{N}
_{0}\backslash \Lambda _{1}}$ $($a subsequence of $\sigma \left( B\right) ),$
and $\left\{ \lambda _{n}^{\infty }\right\} _{n\in
\mathbb{N}
_{0}\backslash \Lambda _{1}^{\infty }}$ $($a subsequence of $\sigma \left(
B^{\infty }\right) ),$ where $\#\Lambda _{1}+\#\Lambda _{1}^{\infty }=$ $%
\left[ \frac{m+2}{2}\right] .$
\end{corollary}

\begin{corollary}
\label{corollary}Assume that $q\ $is $C^{m}$ near $\pi \ $and$\ $the values
of $q^{\left( j\right) }\left( \pi \right) ,$ $j=0,1,\ldots ,m,$\ are known
a priori. Then $h,$ $\beta ,$ $\gamma $ and $q$ on $\left[ 0,\pi \right] $
can be uniquely determined by the following information $(1)$ or $(2):$

$(1)$ all the eigenvalues $\left\{ \lambda _{n}\right\} _{n\in
\mathbb{N}
_{0}}$ of $B$ and a subsequence of the normalizing constants $\left\{ \alpha
_{n+\nu }\right\} _{n\in S_{B}\backslash \Lambda }^{\nu =0,1,\ldots ,k_{n}},$
where $0\leq k_{n}\leq m_{n}-1,$ $\Lambda \subset S_{B}$ and $%
\sum\limits_{n\in \Lambda }m_{n}+\sum\limits_{n\in S_{B}\backslash \Lambda
}\left( m_{n}-k_{n}-1\right) =\left[ \frac{m+3}{2}\right] ;$

$(2)$ all the eigenvalues $\left\{ \lambda _{n}^{\infty }\right\} _{n\in
\mathbb{N}
_{0}}$ of $B^{\infty }$ and a subsequence of the normalizing constants $%
\left\{ \alpha _{n+\nu }^{\infty }\right\} _{n\in s_{B^{\infty }}\backslash
\Lambda ^{\infty }}^{\nu =0,1,\ldots ,k_{n}^{\infty }},$ where $0\leq
k_{n}^{\infty }\leq m_{n}^{\infty }-1,$ $\Lambda ^{\infty }\subset
S_{B^{\infty }}$ and $\sum\limits_{n\in \Lambda ^{\infty }}m_{n}^{\infty
}+\sum\limits_{n\in S_{B^{\infty }}\backslash \Lambda ^{\infty }}\left(
m_{n}^{\infty }-k_{n}^{\infty }-1\right) =\left[ \frac{m+1}{2}\right] .$
\end{corollary}

\begin{remark}
Suppose that {$b_{1}=\dfrac{\beta +\beta ^{-1}}{2}$ is} known a priori. Then
from $\left( \ref{delta1}\right) $, $\left( \ref{delta2}\right) ,$ $\left( %
\ref{ch1}\right) $ and $\left( \ref{ch2}\right) $, one deduces that $\Delta
\left( \lambda \right) $ and $\Delta ^{\infty }\left( \lambda \right) $ can
be uniquely determined by $\sigma \left( B\right) $ and $\sigma \left(
B^{\infty }\right) ,$ respectively$;$ thus by $\left( \ref{24}\right) \ $and
$\left( \ref{63}\right) ,$ we know that Corollary \ref{corollary} remains
valid if the conditions on the normalizing constants $\left\{ \alpha _{n+\nu
}\right\} _{n\in S_{B}\backslash \Lambda }^{\nu =0,1,\ldots ,k_{n}}$ and $%
\left\{ \alpha _{n+\nu }^{\infty }\right\} _{n\in S_{B^{\infty }}\backslash
\Lambda ^{\infty }}^{\nu =0,1,\ldots ,k_{n}^{\infty }}$ are replaced by the
conditions on the ratios $\left\{ \kappa _{n+\nu }\right\} _{n\in
S_{B}\backslash \Lambda }^{\nu =0,1,\ldots ,k_{n}}$ and $\left\{ \kappa
_{n+\nu }^{\infty }\right\} _{n\in S_{B^{\infty }}\backslash \Lambda
^{\infty }}^{\nu =0,1,\ldots ,k_{n}^{\infty }},$ respectively.

\begin{corollary}
\label{corollary copy(2)}Let $d\in \left( 0,\frac{\pi }{2}\right).$ Assume
that $q$ is $C^{m}$ near $\frac{\pi }{2}\ $and $q\ $on $\left[ \frac{\pi }{2}%
,\pi \right] \ $are known a priori. Then $h,$ $\beta ,$ $\gamma $ and $q$ on
$\left[ 0,\pi \right] $ can be uniquely determined by all the eigenvalues $%
\left\{ \lambda _{n}\right\} _{n\in
\mathbb{N}
_{0}}$ of $B$ except for $\left( \left[ \frac{m+2}{2}\right] \right) ,$ or
all the eigenvalues $\left\{ \lambda _{n}^{\infty }\right\} _{n\in
\mathbb{N}
_{0}}$ of $B^{\infty }$ except for $\left( \left[ \frac{m+1}{2}\right]
\right) $.
\end{corollary}
\end{remark}

To prove Theorem \ref{theorem}, we first give a lemma on $F\left( \lambda
\right) $ defined by $\left( \ref{FFF}\right) .$

\begin{lemma}
\label{iy}Assume that $q,$ $\widetilde{q}\in $ $C^{m}\ $near $b\in \left(
d,\pi \right] ,$ $q$ $=\tilde{q}$ a.e. on $\left[ b,\pi \right] $ $($in
particular, for $b=\pi :$ $q^{\left( j\right) }(\pi )=\widetilde{q}^{\left(
j\right) }(\pi )$ for $j=0,1,\ldots ,m).$ Then one observes that
\begin{equation*}
\left\vert F\left( iy\right) \right\vert =o\left( \left\vert y\right\vert ^{-%
\frac{m+1}{2}}\exp \left( 2\left\vert \text{Im}\sqrt{iy}\right\vert b\right)
\right) \ \text{as }y\text{ }\left( \text{real}\right) \rightarrow \infty .
\end{equation*}
\end{lemma}

\begin{proof}
Recall Definition $\ref{yyy}$ (in the Appendix) for the functions $%
y_{i,d}(x,\lambda )$ and $\widetilde{y}_{i,d}(x,\lambda ),$ $i=1,2.$ Then
from Lemma \ref{Fqh} we know that for $b\in \left( d,\pi \right] ,$%
\begin{eqnarray*}
F\left( \lambda \right) &=&\varphi \left( b,\lambda \right) \widetilde{%
\varphi }^{\prime }\left( b,\lambda \right) -\varphi ^{\prime }\left(
b,\lambda \right) \widetilde{\varphi }\left( b,\lambda \right) \\
&=&\left[ \beta \varphi \left( d-0,\lambda \right) y_{1,d}(b,\lambda
)+\left( \beta ^{-1}\varphi ^{\prime }\left( d-0,\lambda \right) +\gamma
\varphi \left( d-0,\lambda \right) \right) y_{2,d}(b,\lambda )\right] \times
\\
&&\left[ \widetilde{\beta }\widetilde{\varphi }\left( d-0,\lambda \right)
\widetilde{y}_{1,d}^{\prime }(b,\lambda )+\left( \widetilde{\beta }^{-1}%
\widetilde{\varphi }^{\prime }\left( d-0,\lambda \right) +\widetilde{\gamma }%
\widetilde{\varphi }\left( d-0,\lambda \right) \right) \widetilde{y}%
_{2,d}^{\prime }(b,\lambda )\right] \\
&&-\left[ \beta \varphi \left( d-0,\lambda \right) y_{1,d}^{\prime
}(b,\lambda )+\left( \beta ^{-1}\varphi ^{\prime }\left( d-0,\lambda \right)
+\gamma \varphi \left( d-0,\lambda \right) \right) y_{2,d}^{\prime
}(b,\lambda )\right] \times \\
&&\left[ \widetilde{\beta }\widetilde{\varphi }\left( d-0,\lambda \right)
\widetilde{y}_{1,d}(b,\lambda )+\left( \widetilde{\beta }^{-1}\widetilde{%
\varphi }^{\prime }\left( d-0,\lambda \right) +\widetilde{\gamma }\widetilde{%
\varphi }\left( d-0,\lambda \right) \right) \widetilde{y}_{2,d}(b,\lambda )%
\right] \\
&=&A_{1}(\lambda )\left[ y_{1,d}(b,\lambda )\widetilde{y}_{1,d}^{\prime
}(b,\lambda )-y_{1d}^{\prime }(b,\lambda )\widetilde{y}_{1,d}(b,\lambda )%
\right] \\
&&+A_{2}(\lambda )\left[ y_{1,d}(b,\lambda )\widetilde{y}_{2,d}^{\prime
}(b,\lambda )-y_{1,d}^{\prime }(b,\lambda )\widetilde{y}_{2,d}(b,\lambda )%
\right] \\
&&+A_{3}(\lambda )\left[ \widetilde{y}_{1,d}^{\prime }(b,\lambda
)y_{2,d}(b,\lambda )-\widetilde{y}_{1,d}(b,\lambda )y_{2,d}^{\prime
}(b,\lambda )\right] \\
&&+A_{4}(\lambda )\left[ y_{2,d}(b,\lambda )\widetilde{y}_{2,d}^{\prime
}(b,\lambda )-y_{2,d}^{\prime }(b,\lambda )\widetilde{y}_{2,d}(b,\lambda )%
\right] ,
\end{eqnarray*}%
where
\begin{eqnarray*}
A_{1}(\lambda ) &=&\beta \widetilde{\beta }\varphi \left( d-0,\lambda
\right) \widetilde{\varphi }\left( d-0,\lambda \right) =O\left( \exp \left(
2\left\vert \text{Im}\sqrt{\lambda }\right\vert d\right) \right) , \\
A_{2}(\lambda ) &=&\beta \varphi \left( d-0,\lambda \right) \left(
\widetilde{\beta }^{-1}\widetilde{\varphi }^{\prime }\left( d-0,\lambda
\right) +\widetilde{\gamma }\widetilde{\varphi }\left( d-0,\lambda \right)
\right) \\
&=&O\left( \sqrt{\left\vert \lambda \right\vert }\exp \left( 2\left\vert
\text{Im}\sqrt{\lambda }\right\vert d\right) \right) , \\
A_{3}(\lambda ) &=&\widetilde{\beta }\widetilde{\varphi }\left( d-0,\lambda
\right) \left( \beta ^{-1}\varphi ^{\prime }\left( d-0,\lambda \right)
+\gamma \varphi \left( d-0,\lambda \right) \right) \\
&=&O\left( \sqrt{\left\vert \lambda \right\vert }\exp \left( 2\left\vert
\text{Im}\sqrt{\lambda }\right\vert d\right) \right) , \\
A_{4}(\lambda ) &=&\left( \widetilde{\beta }^{-1}\widetilde{\varphi }%
^{\prime }\left( d-0,\lambda \right) +\widetilde{\gamma }\widetilde{\varphi }%
\left( d-0\right) \right) \left( \beta ^{-1}\varphi ^{\prime }\left(
d-0,\lambda \right) +\gamma \varphi \left( d-0,\lambda \right) \right) \\
&=&O\left( \left\vert \lambda \right\vert \exp \left( 2\left\vert \text{Im}%
\sqrt{\lambda }\right\vert d\right) \right) .
\end{eqnarray*}%
as $\left\vert \lambda \right\vert \rightarrow \infty .$ Note that the
asymptotics of $A_{1},$ $A_{2},$ $A_{3}$ and $A_{4}$ can be directly
obtained by Lemma \ref{jjgj}. Hence from Proposition \ref{ooo copy(1)} it
follows that as $y$ $\left( \text{real}\right) \rightarrow \infty ,$
\begin{eqnarray*}
&&\left\vert F\left( iy\right) \right\vert \\
&\leq &\left\vert A_{1}(iy)\right\vert o\left( \frac{\exp \left( 2\left\vert
\text{Im}\sqrt{iy}\right\vert \left( b-d\right) \right) }{\left\vert \sqrt{iy%
}\right\vert ^{m+1}}\right) +\left\vert A_{2}(iy)\right\vert o\left( \frac{%
\exp \left( 2\left\vert \text{Im}\sqrt{iy}\right\vert \left( b-d\right)
\right) }{\left\vert \sqrt{iy}\right\vert ^{m+2}}\right) \\
&&+\left\vert A_{3}(iy)\right\vert o\left( \frac{\exp \left( 2\left\vert
\text{Im}\sqrt{iy}\right\vert \left( b-d\right) \right) }{\left\vert \sqrt{iy%
}\right\vert ^{m+2}}\right) +\left\vert A_{4}(iy)\right\vert o\left( \frac{%
\exp \left( 2\left\vert \text{Im}\sqrt{iy}\right\vert \left( b-d\right)
\right) }{\left\vert \sqrt{iy}\right\vert ^{m+3}}\right) \\
&=&o\left( \left\vert y\right\vert ^{-\frac{m+1}{2}}\exp \left( 2\left\vert
\text{Im}\sqrt{iy}\right\vert b\right) \right) .
\end{eqnarray*}%
This completes the proof.
\end{proof}

Now we turn to prove Theorem \ref{theorem}.

\begin{proof}[Proof of Theorem \protect\ref{theorem}]
For any $\lambda _{n}\in \widehat{W}\ $and $\lambda _{n}^{\infty }\in
\widehat{W^{\infty }}\ $where $n\in S_{B}\ $and $n\in S_{B^{\infty }},$ let $%
\gamma _{n}$ and $\gamma _{n}^{\infty }$ denote the number of occurrences of
$\lambda _{n}$ in $W\ $and $\lambda _{n}^{\infty }$ in $W^{\infty },$
respectively$.$ Denote
\begin{equation}
H\left( \lambda \right) :=\frac{F\left( \lambda \right) }{G_{\Xi }\left(
\lambda \right) },  \label{Hlamuda}
\end{equation}%
where
\begin{equation}
G_{\Xi }\left( \lambda \right) :=G_{W}\left( \lambda \right) G_{W_{1}}\left(
\lambda \right) G_{W^{\infty }}\left( \lambda \right) G_{W_{1}^{\infty
}}\left( \lambda \right) ,  \label{G3}
\end{equation}%
\begin{eqnarray*}
&&G_{W}\left( \lambda \right) :=\prod\limits_{\lambda _{n}\in \widehat{W}%
,n\in S_{B}}\left( 1-\frac{\lambda }{\lambda _{n}}\right) ^{\gamma
_{n}},G_{W_{1}}\left( \lambda \right) :=\prod\limits_{\lambda _{n}\in
\widehat{W_{1}},n\in S_{B}}\left( 1-\frac{\lambda }{\lambda _{n}}\right)
^{k_{n}}, \\
&&G_{W^{\infty }}\left( \lambda \right) :=\prod\limits_{\lambda _{n}^{\infty
}\in \widehat{W^{\infty }},n\in S_{B^{\infty }}}\left( 1-\frac{\lambda }{%
\lambda _{n}^{\infty }}\right) ^{\gamma _{n}^{\infty }},G_{W_{1}^{\infty
}}\left( \lambda \right) :=\prod\limits_{\lambda _{n}^{\infty }\in \widehat{%
W_{1}^{\infty }},n\in S_{B^{\infty }}}\left( 1-\frac{\lambda }{\lambda
_{n}^{\infty }}\right) ^{k_{n}^{\infty }}.
\end{eqnarray*}%
Then it follows from $\left( \ref{hhypo1}\right) ,$ $\left( \ref{hhypo2}%
\right) ,$ Lemma \ref{F}, and the fact $\widehat{\sigma \left( B\right) }%
\cap \widehat{\sigma \left( B^{\infty }\right) }=\emptyset $ that $H\left(
\lambda \right) $ is an entire function. From Lemma \ref{Fqh}, we know that $%
F\left( \lambda \right) $ is an entire function of order less than $\frac{1}{%
2};$ $\Delta \left( \lambda \right) $ and $\Delta ^{\infty }\left( \lambda
\right) $ are entire functions of order $\frac{1}{2}.$ Moreover, since the
order of canonical product of an entire function is equal to its convergence
exponent of zeros (\cite[P16]{levision}), we can obtain that $G_{\Xi }\left(
\lambda \right) $ is an entire function of order less than $\frac{1}{2},$
and so the order of $H\left( \lambda \right) $ is at most $\frac{1}{2}.$

Now we aim to prove that $H\left( \lambda \right) \equiv 0.$ By Lemma \ref%
{proposition}, it is sufficient to prove that $\left\vert H\left( iy\right)
\right\vert \rightarrow 0$ as $y$ $\left( \text{real}\right) \rightarrow
\infty .$ From Lemma \ref{number} and the assumption $\left( \ref{hypo}%
\right) $, we know that there exists a constant $M>0$ such that
\begin{equation*}
\left\vert G_{\Xi }\left( iy\right) \right\vert \geq M\left\vert
y\right\vert ^{-\frac{m+1}{2}}\exp \left( 2\left\vert \text{Im}\sqrt{iy}%
\right\vert b\right) ,
\end{equation*}%
and thus according to $\left( \ref{Hlamuda}\right) $ and Lemma \ref{iy}, one
has
\begin{equation*}
\left\vert H\left( iy\right) \right\vert \leq \left\vert \frac{o\left(
\left\vert y\right\vert ^{-\frac{m+1}{2}}\exp \left( 2\left\vert \text{Im}%
\sqrt{iy}\right\vert b\right) \right) }{M\left\vert y\right\vert ^{-\frac{m+1%
}{2}}\exp \left( 2\left\vert \text{Im}\sqrt{iy}\right\vert b\right) }%
\right\vert =o\left( 1\right) \text{ as }y\text{ }\left( \text{real}\right)
\rightarrow \infty .
\end{equation*}%
This implies that $H\left( \lambda \right) \equiv 0\ $and thus $F\left(
\lambda \right) \equiv 0.$ Then we conclude from Lemma \ref{unique by F}
that $h=\widetilde{h},$ $\beta =\widetilde{\beta },$ $\gamma =\widetilde{%
\gamma }$ and $q=\widetilde{q}$ a.e. on $\left[ 0,\pi \right] .$
\end{proof}

\subsubsection{Pairs of Eigenvalues and Ratios}

\begin{Hypothesis}
\label{hypo2}Consider the subsequences $W$ and $W^{\infty }$ satisfying
\begin{equation*}
W<<\sigma \left( B\right) ,W<<\sigma \left( \widetilde{B}\right) ,\text{ }%
W^{\infty }<<\sigma \left( B^{\infty }\right) ,W^{\infty }<<\sigma \left(
\widetilde{B}^{\infty }\right)
\end{equation*}%
and the following conditions:

$(1)$ for any $\lambda _{n}=\widetilde{\lambda }_{\widetilde{n}}\in \widehat{%
W}\ $where $n\in S_{B}\ $and $\widetilde{n}\in S_{\widetilde{B}},$ suppose
that%
\begin{equation}
\text{ }\kappa _{n+\nu }=\widetilde{\kappa }_{\widetilde{n}+\nu }\ \text{for
}\nu =0,1,\ldots ,k_{n}-1,\text{ }  \label{hhypo 3}
\end{equation}%
where $k_{n}$ equals the number of occurrences of $\lambda _{n}\ $in $W;$

$\left( 2\right) $ for any $\lambda _{n}^{\infty }=\widetilde{\lambda }_{%
\widetilde{n}}^{\infty }\in \widehat{W^{\infty }}$ where $n\in S_{B^{\infty
}}\ $and $\widetilde{n}\in S_{\widetilde{B}^{\infty }},$ suppose that
\begin{equation}
\kappa _{n+\gamma }^{\infty }=\widetilde{\kappa }_{\widetilde{n}+\gamma
}^{\infty }\ \text{for }\gamma =0,1,\ldots ,k_{n}^{\infty }-1,\text{ }
\label{hhypo4}
\end{equation}%
where $k_{n}^{\infty }$ equals the number of occurrences of $\lambda
_{n}^{\infty }\ $in $W^{\infty }.$
\end{Hypothesis}

\begin{theorem}
\label{theorem copy(1)}Assume Hypothesis \ref{hypo2} and suppose that $q=%
\widetilde{q}$ a.e. on $\left[ b,\pi \right] ,$%
\begin{equation}
N_{W}(t)+N_{W^{\infty }}(t)\geq AN_{\sigma \left( B\right) }(t)+\left( \frac{%
b}{\pi }-A\right) N_{\sigma \left( B^{\infty }\right) }(t)-\frac{A}{2}%
+\epsilon  \label{hypo4}
\end{equation}%
for sufficiently large $t\in
\mathbb{R}
,$ where $\epsilon $ is an arbitrary positive constant$.$ Then $h=\widetilde{%
h},$ $\beta =\widetilde{\beta },$ $\gamma =\widetilde{\gamma }$ and $q=%
\widetilde{q}$ a.e. on $\left[ 0,\pi \right] .$
\end{theorem}

\begin{proof}
Denote%
\begin{equation}
H_{1}\left( \lambda \right) :=\frac{F_{1}\left( \lambda \right) }{G_{\Theta
}\left( \lambda \right) },\text{ }H_{2}\left( \lambda \right) =\frac{%
F_{2}\left( \lambda \right) }{G_{\Theta }\left( \lambda \right) },
\label{FGFG}
\end{equation}%
where $G_{\Theta }\left( \lambda \right) :=G_{W}\left( \lambda \right)
G_{W^{\infty }}\left( \lambda \right) ,$
\begin{equation}
G_{W}\left( \lambda \right) :=\prod\limits_{\lambda _{n}\in \widehat{W},n\in
S_{B}}\left( 1-\frac{\lambda }{\lambda _{n}}\right) ^{k_{n}},\text{ }%
G_{W^{\infty }}\left( \lambda \right) :=\prod\limits_{\lambda _{n}^{\infty
}\in \widehat{W^{\infty }},n\in S_{B^{\infty }}}\left( 1-\frac{\lambda }{%
\lambda _{n}^{\infty }}\right) ^{k_{n}^{\infty }}  \label{GG}
\end{equation}%
and%
\begin{equation*}
F_{1}\left( \lambda \right) :=\varphi \left( b,\lambda \right) -\widetilde{%
\varphi }\left( b,\lambda \right) ,\text{ }F_{2}\left( \lambda \right)
:=\varphi ^{\prime }\left( b,\lambda \right) -\widetilde{\varphi }^{\prime
}\left( b,\lambda \right) .
\end{equation*}

\textbf{Step 1}: This step is devoted to show that $H_{1}\left( \lambda
\right) $ and $H_{2}\left( \lambda \right) $ are entire functions of $%
\lambda \in
\mathbb{C}
$. We first prove that $\frac{F_{1}\left( \lambda \right) }{G_{W}\left(
\lambda \right) }\ $and $\frac{F_{2}\left( \lambda \right) }{G_{W}\left(
\lambda \right) }$ are entire functions of $\lambda \in
\mathbb{C}
$.\ In fact, from $\left( \ref{ratios}\right) ,$ $\left( \ref{7}\right) ,$ $%
\left( \ref{3f}\right) ,$ $\left( \ref{hhypo 3}\right) $, $H=\widetilde{H}$
and $q=\widetilde{q}$ a.e. on $\left[ b,\pi \right] ,$ one can easily deduce
that for $\lambda _{n}\in \widehat{W},$ $n\in S_{B}$,
\begin{equation*}
\varphi _{\nu }\left( x,\lambda _{n}\right) =\widetilde{\varphi }_{\nu
}\left( x,\lambda _{n}\right) ,\ x\in \left[ b,\pi \right] \text{,}
\end{equation*}%
where $\nu =0,1,\ldots ,k_{n}-1.$ Thus for $\lambda _{n}\in \widehat{W},$ $%
n\in S_{B}$, $\nu =0,1,\ldots ,k_{n}-1,$ one observes that
\begin{equation}
\varphi _{\nu }\left( b,\lambda _{n}\right) =\widetilde{\varphi }_{\nu
}\left( b,\lambda _{n}\right) ,\text{ }\varphi _{\nu }^{\prime }\left(
b,\lambda _{n}\right) =\widetilde{\varphi }_{\nu }^{\prime }\left( b,\lambda
_{n}\right) ,\text{ }  \label{fainiu}
\end{equation}%
and thus
\begin{eqnarray}
\left. \frac{d^{\nu }F_{1}\left( \lambda \right) }{d\lambda ^{\nu }}%
\right\vert _{\lambda =\lambda _{n}} &:&=\nu !\left( \varphi _{\nu }\left(
b,\lambda _{n}\right) -\widetilde{\varphi }_{\nu }\left( b,\lambda
_{n}\right) \right) =0,  \label{555} \\
\left. \frac{d^{\nu }F_{2}\left( \lambda \right) }{d\lambda ^{\nu }}%
\right\vert _{\lambda =\lambda _{n}} &:&=\nu !\left( \varphi _{\nu }^{\prime
}\left( b,\lambda _{n}\right) -\widetilde{\varphi }_{\nu }^{\prime }\left(
b,\lambda _{n}\right) \right) =0.  \label{gggpie}
\end{eqnarray}%
Then in view of $\left( \ref{FGFG}\right) $ and $\left( \ref{GG}\right) ,$
we infer that $\frac{F_{1}\left( \lambda \right) }{G_{W}\left( \lambda
\right) }\ $and $\frac{F_{2}\left( \lambda \right) }{G_{W}\left( \lambda
\right) }$ are entire functions of $\lambda \in
\mathbb{C}
$. Similarly, we can also prove that $\frac{F_{1}\left( \lambda \right) }{%
G_{W^{\infty }}\left( \lambda \right) }\ $and $\frac{F_{2}\left( \lambda
\right) }{G_{W^{\infty }}\left( \lambda \right) }$ are entire functions of $%
\lambda \in
\mathbb{C}
$. Therefore, from the fact $\widehat{\sigma \left( B\right) }\cap \widehat{%
\sigma \left( B^{\infty }\right) }=\emptyset $ we conclude that $H_{1}\left(
\lambda \right) $ and $H_{2}\left( \lambda \right) $ are entire functions of
$\lambda \in
\mathbb{C}
$. Furthermore, it is easy to see that the order of $H_{1}\left( \lambda
\right) $ and $H_{2}\left( \lambda \right) $ are less than $\frac{1}{2}$.

\textbf{Step 2}: Now we want to use Lemma \ref{proposition} to prove $%
H_{1}\left( \lambda \right) \equiv 0.$ From Lemma \ref{number} and the
assumption $\left( \ref{hypo4}\right) ,$ it follows that there exists a
constant $M>0$ such that
\begin{equation}
\left\vert G_{\Theta }\left( iy\right) \right\vert \geq M\left\vert
y\right\vert ^{\epsilon }\exp \left( \left\vert \text{Im}\sqrt{iy}%
\right\vert b\right) .  \label{GS}
\end{equation}%
Moreover, from $\left( \ref{1}\right) $ we know that
\begin{equation*}
F_{1}\left( \lambda \right) =\left( \left( b_{1}-\widetilde{b}_{1}\right)
\cos \left( \sqrt{\lambda }b\right) +\left( b_{2}-\widetilde{b}_{2}\right)
\cos \left( \sqrt{\lambda }\left( 2d-b\right) \right) \right) +O\left( \frac{%
\exp \left( \left\vert \text{Im}\sqrt{\lambda }\right\vert b\right) }{\sqrt{%
\lambda }}\right) ,
\end{equation*}%
and thus%
\begin{equation}
\left\vert F_{1}\left( iy\right) \right\vert =\exp \left( \left\vert \text{Im%
}\sqrt{iy}\right\vert b\right) \left( \frac{\left\vert b_{1}-\widetilde{b}%
_{1}\right\vert }{2}+o\left( 1\right) \right) \text{ as }y\text{ }\left(
\text{real}\right) \rightarrow \infty .  \label{f1iy}
\end{equation}%
Therefore, by $\left( \ref{FGFG}\right) ,$ $\left( \ref{GS}\right) $ and $%
\left( \ref{f1iy}\right) ,$ one deduces that
\begin{equation*}
\left\vert H_{1}\left( iy\right) \right\vert \leq \left\vert \frac{\exp
\left( \left\vert \text{Im}\sqrt{iy}\right\vert b\right) \left( \frac{%
\left\vert b_{1}-\widetilde{b}_{1}\right\vert }{2}+o\left( 1\right) \right)
}{M\left\vert y\right\vert ^{\epsilon }\exp \left( \left\vert \text{Im}\sqrt{%
iy}\right\vert b\right) }\right\vert =O\left( y^{-\epsilon }\right) ,
\end{equation*}%
as $y$ $\left( \text{real}\right) \rightarrow \infty .$ By Lemma \ref%
{proposition}, one deduces that $H_{1}\left( \lambda \right) \equiv 0\ $and
therefore $F_{1}\left( \lambda \right) \equiv 0$ for all $\lambda \in
\mathbb{C}
,$ i.e., $\varphi \left( b,\lambda \right) \equiv \widetilde{\varphi }\left(
b,\lambda \right) .$

\textbf{Step 3}: From the fact $\varphi \left( b,\lambda \right) \equiv
\widetilde{\varphi }\left( b,\lambda \right) $, we know that
\begin{equation*}
H_{2}\left( \lambda \right) =\frac{\left[ \varphi ^{\prime }\left( b,\lambda
\right) -\widetilde{\varphi }^{\prime }\left( b,\lambda \right) \right]
\varphi \left( b,\lambda \right) }{G_{\Theta }\left( \lambda \right) \varphi
\left( b,\lambda \right) }=\frac{-F\left( \lambda \right) }{G_{\Theta
}\left( \lambda \right) \varphi \left( b,\lambda \right) }.
\end{equation*}%
Hence, from $\left( \ref{faib}\right) ,$ $\left( \ref{GS}\right) \ $and
Lemma \ref{iy}, we have
\begin{eqnarray*}
\left\vert H_{2}\left( iy\right) \right\vert &\leq &\left\vert \frac{o\left(
\exp \left( 2\left\vert \text{Im}\sqrt{iy}\right\vert b\right) \right) }{%
M\left\vert y\right\vert ^{\epsilon }\exp \left( \left\vert \text{Im}\sqrt{iy%
}\right\vert b\right) \frac{b_{1}}{2}\exp \left( \left\vert \text{Im}\sqrt{iy%
}\right\vert b\right) \left( 1+o\left( 1\right) \right) }\right\vert \\
&=&o\left( y^{-\epsilon }\right) .
\end{eqnarray*}%
Then it follows from Lemma \ref{proposition} that $H_{2}\left( \lambda
\right) =0$ and thus $F\left( \lambda \right) \equiv 0$ for all $\lambda \in
\mathbb{C}
.$ Now we can conclude from Lemma \ref{unique by F} that $h=\widetilde{h},$ $%
\beta =\widetilde{\beta },$ $\gamma =\widetilde{\gamma }$ and $q=\widetilde{q%
}$ a.e. on $\left[ 0,\pi \right] .$ The proof is thus completed.
\end{proof}

\begin{remark}
If $b_{1}=\frac{\beta +\beta ^{-1}}{2}$ is given, then it is easy to see
from $\left( \ref{f1iy}\right) $ that
\begin{equation*}
\left\vert F_{1}\left( iy\right) \right\vert =o\left( \exp \left( \left\vert
\text{Im}\sqrt{iy}\right\vert b\right) \right) \text{ as }y\text{ }\left(
\text{real}\right) \rightarrow \infty .
\end{equation*}%
In this case the assumption $\left( \ref{hypo4}\right) $ in Theorem \ref%
{theorem copy(1)} can be replaced by
\begin{equation*}
N_{W}(t)+N_{W^{\infty }}(t)\geq AN_{\sigma \left( B\right) }(t)+\left( \frac{%
b}{\pi }-A\right) N_{\sigma \left( B^{\infty }\right) }(t)-\frac{A}{2}.
\end{equation*}
\end{remark}

\subsection{Case II: $q$ is known on $\left[ b,\protect\pi \right] ,$ where $%
b=d$}

\subsubsection{Pairs of Eigenvalues and Normalizing Constants}

\begin{theorem}
\label{theorem copy(2)}Assume Hypothesis \ref{hypo1} and suppose that $q=%
\widetilde{q}$ a.e. on $\left[ b,\pi \right] ,$ and
\begin{eqnarray}
&&N_{W}(t)+N_{W_{1}}(t)+N_{W^{\infty }}(t)+N_{W_{1}^{\infty }}(t)
\label{hypo6} \\
&\geq &AN_{\sigma \left( B\right) }(t)+\left( \frac{2d}{\pi }-A\right)
N_{\sigma \left( B^{\infty }\right) }(t)-\frac{A}{2}+\frac{1}{2}+\epsilon
\notag
\end{eqnarray}%
for sufficiently large $t\in
\mathbb{R}
,$ where $\epsilon $ is an arbitrary positive constant$.$ Then $h=\widetilde{%
h},$ $\beta =\widetilde{\beta },$ $\gamma =\widetilde{\gamma }$ and $q=%
\widetilde{q}$ a.e. on $\left[ 0,\pi \right] .$
\end{theorem}

\begin{proof}
Let
\begin{equation}
H\left( \lambda \right) :=\frac{F\left( \lambda \right) }{G_{\Xi }\left(
\lambda \right) },  \label{hhhh}
\end{equation}%
where $G_{\Xi }\left( \lambda \right) $ is similarly defined as in $\left( %
\ref{G3}\right) $ and $F\left( \lambda \right) $ is defined by $\left( \ref%
{FFF}\right) .$ By Lemma \ref{Fqh} we know that if $b=d$,
\begin{eqnarray}
F\left( \lambda \right) &=&\left. \left\langle \varphi \left( x,\lambda
\right) ,\widetilde{\varphi }\left( x,\lambda \right) \right\rangle
\right\vert _{x=d+0}  \label{flamuda2} \\
&=&\beta \widetilde{\beta }^{-1}\varphi \left( d-0,\lambda \right)
\widetilde{\varphi }^{\prime }\left( d-0,\lambda \right) -\beta ^{-1}%
\widetilde{\beta }\widetilde{\varphi }\left( d-0,\lambda \right) \varphi
^{\prime }\left( d-0,\lambda \right)  \notag \\
&&+\widetilde{\gamma }\beta \varphi \left( d-0,\lambda \right) \widetilde{%
\varphi }\left( d-0,\lambda \right) -\gamma \widetilde{\beta }\widetilde{%
\varphi }\left( d-0,\lambda \right) \varphi \left( d-0,\lambda \right) .
\notag
\end{eqnarray}%
Moreover, from Lemma \ref{jjgj} it is easy to see that
\begin{eqnarray}
\left\vert \varphi \left( d-0,iy\right) \right\vert &=&\frac{1}{2}\exp
\left( \left\vert \text{Im}\sqrt{iy}\right\vert d\right) \left( 1+o\left(
1\right) \right) ,  \label{faiiy} \\
\left\vert \varphi ^{\prime }\left( d-0,iy\right) \right\vert &=&\frac{1}{2}%
\left\vert y\right\vert ^{\frac{1}{2}}\exp \left( \left\vert \text{Im}\sqrt{%
iy}\right\vert d\right) \left( 1+o\left( 1\right) \right)  \notag
\end{eqnarray}%
as $y$ $\left( \text{real}\right) \rightarrow \infty ,$ and hence%
\begin{equation}
\left\vert F\left( iy\right) \right\vert =O\left( \left\vert y\right\vert ^{%
\frac{1}{2}}\exp \left( 2\left\vert \text{Im}\sqrt{iy}\right\vert d\right)
\right) \text{ as }y\text{ }\left( \text{real}\right) \rightarrow \infty .
\label{FIY}
\end{equation}%
By Lemma \ref{number} and $\left( \ref{hypo6}\right) ,$ we infer that there
exists a constant $M>0$ such that
\begin{equation}
\left\vert G_{\Xi }\left( iy\right) \right\vert \geq M\left\vert
y\right\vert ^{\frac{1}{2}+\epsilon }\exp \left( 2\left\vert \text{Im}\sqrt{%
iy}\right\vert d\right) .  \label{Gttt}
\end{equation}%
Therefore, from $\left( \ref{hhhh}\right) ,$ $\left( \ref{FIY}\right) \ $and
$\left( \ref{Gttt}\right) ,$ we have
\begin{equation*}
\left\vert H\left( iy\right) \right\vert \leq \left\vert \frac{O\left(
\left\vert y\right\vert ^{\frac{1}{2}}\exp \left( 2\left\vert \text{Im}\sqrt{%
iy}\right\vert d\right) \right) }{M\left\vert y\right\vert ^{\frac{1}{2}%
+\epsilon }\exp \left( 2\left\vert \text{Im}\sqrt{iy}\right\vert d\right) }%
\right\vert =O\left( \left\vert y\right\vert ^{-\epsilon }\right)
\end{equation*}%
as $y$ $\left( \text{real}\right) \rightarrow \infty .$ This implies that $%
H\left( \lambda \right) \equiv 0\ $and hence $F\left( \lambda \right) \equiv
0$ for all $\lambda \in
\mathbb{C}
$ by the argument of the proof of Theorem \ref{theorem}$.$ Then the
statement of this theorem can be concluded from Lemma \ref{unique by F}.
\end{proof}

\begin{remark}
\label{theorem copy(6)}$(1)$ If $\beta =\widetilde{\beta },$ instead of
condition $\left( \ref{hypo6}\right) $, we only need the following condition$%
:$%
\begin{eqnarray*}
&&N_{S}(t)+N_{S_{1}}(t)+N_{S_{1}^{\infty }}(t)+N_{S^{\infty }}(t) \\
&\geq &AN_{\sigma \left( B\right) }(t)+\left( \frac{2d}{\pi }-A\right)
N_{\sigma \left( B^{\infty }\right) }(t)-\frac{A}{2}+\epsilon ;
\end{eqnarray*}%
$(2)$ If $\beta =\widetilde{\beta },$ $\gamma =\widetilde{\gamma },$ $q,$ $%
\widetilde{q}\in $ $C^{m}$ near $d,$ then, instead of condition $\left( \ref%
{hypo6}\right) $, we only need the following condition$:$%
\begin{eqnarray*}
&&N_{S}(t)+N_{S_{1}}(t)+N_{S_{1}^{\infty }}(t)+N_{S^{\infty }}(t) \\
&\geq &AN_{\sigma \left( B\right) }(t)+\left( \frac{2d}{\pi }-A\right)
N_{\sigma \left( B^{\infty }\right) }(t)-\frac{A}{2}-\frac{m+1}{2}\text{.}
\end{eqnarray*}
\end{remark}

In fact, one notes that for $x\in \left( 0,d\right) ,$
\begin{eqnarray}
&&\varphi \left( x,\lambda \right) \widetilde{\varphi }^{\prime }\left(
x,\lambda \right) -\widetilde{\varphi }\left( x,\lambda \right) \varphi
^{\prime }\left( x,\lambda \right)  \label{faifai} \\
&=&y_{1,0}(x,\lambda )\widetilde{y}_{1,0}^{\prime }(x,\lambda
)-y_{1,0}^{\prime }(x,\lambda )\widetilde{y}_{1,0}(x,\lambda )  \notag \\
&&+h\left( y_{2,0}(x,\lambda )\widetilde{y}_{1,0}^{\prime }(x,\lambda
)-y_{2,0}^{\prime }(x,\lambda )\widetilde{y}_{1,0}(x,\lambda )\right)  \notag
\\
&&+\widetilde{h}\left( y_{1,0}(x,\lambda )\widetilde{y}_{2,0}^{\prime
}(x,\lambda )-\widetilde{y}_{2,0}(x,\lambda )y_{1,0}^{\prime }(x,\lambda
)\right)  \notag \\
&&+h\widetilde{h}\left( y_{2,0}(x,\lambda )\widetilde{y}_{2,0}^{\prime
}(x,\lambda )-y_{2,0}^{\prime }(x,\lambda )\widetilde{y}_{2,0}(x,\lambda
)\right) .  \notag
\end{eqnarray}%
Therefore, if $\beta =\widetilde{\beta },$ it follows from $\left( \ref%
{flamuda2}\right) ,$ $\left( \ref{faiiy}\right) $ and Remark \ref{ooo
copy(4)} that
\begin{equation}
\left\vert F\left( iy\right) \right\vert =O\left( \exp \left( 2\left\vert
\text{Im}\sqrt{iy}\right\vert d\right) \right) \text{ as }y\text{ }\left(
\text{real}\right) \rightarrow \infty .  \label{FIY2}
\end{equation}%
Moreover, if $\beta =\widetilde{\beta },$ $\gamma =\widetilde{\gamma },$ $q,$
$\widetilde{q}\in $ $C^{m}$ near $d,$ it is easy to see from $\left( \ref%
{flamuda2}\right) $ and Proposition \ref{ooo copy(1)} that
\begin{equation}
\left\vert F\left( iy\right) \right\vert =o\left( \left\vert y\right\vert ^{-%
\frac{m+1}{2}}\exp \left( 2\left\vert \text{Im}\sqrt{iy}\right\vert d\right)
\right) \text{ as }y\text{ }\left( \text{real}\right) \rightarrow \infty .
\label{ass}
\end{equation}%
Thus by the argument of the proof of Theorem \ref{theorem copy(2)}, Remark %
\ref{theorem copy(6)} can be directly obtained.

\begin{corollary}
\label{corollary copy(3)}Let $d=\frac{\pi }{2}.$ Assume that $q$ is $C^{m}$
near $\frac{\pi }{2}\ $and suppose that $\beta ,$ $\gamma ,$ $q\ $on $\left[
\frac{\pi }{2},\pi \right] \ $are known a priori. Then $h$ and $q$ on $\left[
0,\pi \right] $ can be uniquely determined by all the eigenvalues $\left\{
\lambda _{n}\right\} _{n\in
\mathbb{N}
_{0}}$ of $B$ except for $\left( \left[ \frac{m+2}{2}\right] \right) ,$ or
all the eigenvalues $\left\{ \lambda _{n}^{\infty }\right\} _{n\in
\mathbb{N}
_{0}}$ of $B^{\infty }$ except for $\left( \left[ \frac{m+1}{2}\right]
\right) $.
\end{corollary}

\begin{corollary}
Let $d=\frac{\pi }{2}.$ Assume that $q$ on $\left[ \frac{\pi }{2},\pi \right]
$ and $\beta $ are known a priori, then $\sigma \left( B\right) $ uniquely
determines $h,$ $\gamma $ and $q$ a.e. on $\left[ 0,\pi \right] .$
\end{corollary}

\subsubsection{Pairs of Eigenvalues and Ratios}

\begin{theorem}
\label{theorem copy(3)}Assume Hypothesis \ref{hypo2} and suppose that $q=%
\widetilde{q}$ a.e. on $\left[ d,\pi \right] ,$%
\begin{equation}
N_{W}(t)+N_{W^{\infty }}(t)\geq AN_{\sigma \left( B\right) }(t)+\left( \frac{%
d}{\pi }-A\right) N_{\sigma \left( B^{\infty }\right) }(t)-\frac{A}{2}+\frac{%
1}{2}+\epsilon  \label{hypo8}
\end{equation}%
for sufficiently large $t\in
\mathbb{R}
,$ where $\epsilon $ is an arbitrary positive constant.$\ $Then $h=%
\widetilde{h},$ $\beta =\widetilde{\beta },$ $\gamma =\widetilde{\gamma }$
and $q=\widetilde{q}$ a.e. on $\left[ 0,\pi \right] .$
\end{theorem}

\begin{proof}
Denote%
\begin{equation}
H_{1}\left( \lambda \right) :=\frac{F_{1}\left( \lambda \right) }{G_{\Theta
}\left( \lambda \right) },\text{ }H_{2}\left( \lambda \right) =\frac{%
F_{2}\left( \lambda \right) }{G_{\Theta }\left( \lambda \right) },
\label{F2}
\end{equation}%
where $G_{\Theta }\left( \lambda \right) :=G_{W}\left( \lambda \right)
G_{W^{\infty }}\left( \lambda \right) ,$
\begin{equation*}
G_{W}\left( \lambda \right) :=\prod\limits_{\lambda _{n}\in \widehat{W},n\in
S_{B}}\left( 1-\frac{\lambda }{\lambda _{n}}\right) ^{k_{n}},\text{ }%
G_{W^{\infty }}\left( \lambda \right) :=\prod\limits_{\lambda _{n}^{\infty
}\in \widehat{W^{\infty }},n\in S_{B^{\infty }}}\left( 1-\frac{\lambda }{%
\lambda _{n}^{\infty }}\right) ^{k_{n}^{\infty }}
\end{equation*}%
and%
\begin{equation*}
F_{1}\left( \lambda \right) :=\varphi \left( d+0,\lambda \right) -\widetilde{%
\varphi }\left( d+0,\lambda \right) ,\text{ }F_{2}\left( \lambda \right)
:=\varphi ^{\prime }\left( d+0,\lambda \right) -\widetilde{\varphi }^{\prime
}\left( d+0,\lambda \right) .
\end{equation*}%
In view of Lemma \ref{number} and $\left( \ref{hypo8}\right) ,$ one has
\begin{equation}
\left\vert G_{\Theta }\left( iy\right) \right\vert \geq M\left\vert
y\right\vert ^{\epsilon +\frac{1}{2}}\exp \left( \left\vert \text{Im}\sqrt{iy%
}\right\vert d\right) .  \label{gs2}
\end{equation}%
In addition, from $\left( \ref{1}\right) \ $it is easy to see that
\begin{equation}
\left\vert F_{1}\left( iy\right) \right\vert =\exp \left( \left\vert \text{Im%
}\sqrt{iy}\right\vert d\right) \left( \frac{\left\vert \beta -\widetilde{%
\beta }\right\vert }{2}+o\left( 1\right) \right) \text{ as }y\text{ }\left(
\text{real}\right) \rightarrow \infty .  \label{F1111}
\end{equation}%
Thus it follows from $\left( \ref{gs2}\right) $ and $\left( \ref{F1111}%
\right) $ that
\begin{eqnarray*}
\left\vert H_{1}\left( iy\right) \right\vert &\leq &\left\vert \frac{\exp
\left( \left\vert \text{Im}\sqrt{iy}\right\vert d\right) \left( \frac{%
\left\vert \beta -\widetilde{\beta }\right\vert }{2}+o\left( 1\right)
\right) }{M\left\vert y\right\vert ^{\epsilon +\frac{1}{2}}\exp \left(
\left\vert \text{Im}\sqrt{iy}\right\vert d\right) }\right\vert \\
&=&O\left( y^{-\epsilon -\frac{1}{2}}\right) \text{ as }y\text{ }\left(
\text{real}\right) \rightarrow \infty .
\end{eqnarray*}%
By a similar proof to that of Theorem $\ref{theorem copy(1)},$ we can obtain
that $H_{1}\left( \lambda \right) \equiv 0,$ and thus $F_{1}\left( \lambda
\right) \equiv 0,$ i.e., $\varphi \left( d+0,\lambda \right) \equiv
\widetilde{\varphi }\left( d+0,\lambda \right) $ for all $\lambda \in
\mathbb{C}
$. Then it follows from $\left( \ref{flamuda2}\right) $ and $\left( \ref{F2}%
\right) $ that
\begin{equation*}
H_{2}\left( \lambda \right) =\frac{\left[ \varphi ^{\prime }\left(
d+0,\lambda \right) -\widetilde{\varphi }^{\prime }\left( d+0,\lambda
\right) \right] \varphi \left( d+0,\lambda \right) }{G_{\Theta }\left(
\lambda \right) \varphi \left( d+0,\lambda \right) }=\frac{-F\left( \lambda
\right) }{G_{\Theta }\left( \lambda \right) \varphi \left( d+0,\lambda
\right) }.
\end{equation*}%
Thus by $\left( \ref{faid}\right) ,$ $\left( \ref{FIY}\right) $ and $\left( %
\ref{gs2}\right) ,$ we infer that as $y$ $\left( \text{real}\right)
\rightarrow \infty ,$%
\begin{eqnarray*}
\left\vert H_{2}\left( iy\right) \right\vert &\leq &\left\vert \frac{O\left(
\left\vert y\right\vert ^{\frac{1}{2}}\exp \left( 2\left\vert \text{Im}\sqrt{%
iy}\right\vert d\right) \right) }{M\left\vert y\right\vert ^{\epsilon +\frac{%
1}{2}}\exp \left( \left\vert \text{Im}\sqrt{iy}\right\vert d\right) \frac{%
\beta }{2}\exp \left( \left\vert \text{Im}\sqrt{iy}\right\vert d\right)
\left( 1+o\left( 1\right) \right) }\right\vert \\
&=&O\left( \left\vert y\right\vert ^{-\epsilon }\right) .
\end{eqnarray*}%
Then by the argument of the proof of Theorem $\ref{theorem copy(1)},$ we can
obtain that $F\left( \lambda \right) \equiv 0.$ Now we conclude from Lemma %
\ref{unique by F} that $h=\widetilde{h},$ $\beta =\widetilde{\beta },$ $%
\gamma =\widetilde{\gamma }$ and $q=\widetilde{q}$ a.e. on $\left[ 0,\pi %
\right] .$
\end{proof}

\begin{remark}
$(1)$ If $\beta \ $is known a priori$,$ then by $\left( \ref{FIY2}\right) $
and $\left( \ref{F1111}\right) $ one has
\begin{equation*}
\left\vert F\left( iy\right) \right\vert =O\left( \exp \left( 2\left\vert
\text{Im}\sqrt{iy}\right\vert d\right) \right) \text{ and }F_{1}\left(
iy\right) =o\left( \exp \left( 2\left\vert \text{Im}\sqrt{iy}\right\vert
d\right) \right)
\end{equation*}%
as $y$ $\left( \text{real}\right) \rightarrow \infty .$ In this case, the
assumption $\left( \ref{hypo8}\right) $ can be replaced by%
\begin{equation*}
N_{W}(t)+N_{W^{\infty }}(t)\geq AN_{\sigma \left( B\right) }(t)+\left( \frac{%
d}{\pi }-A\right) N_{\sigma \left( B^{\infty }\right) }(t)-\frac{A}{2}%
+\epsilon .
\end{equation*}%
$(2)$ If $\beta \ $and $\gamma $ are known a priori$,$ by $\left( \ref{ass}%
\right) $ $\left( \text{for }m=-1\right) $ and $\left( \ref{F1111}\right) $
one has
\begin{equation*}
\left\vert F\left( iy\right) \right\vert =o\left( \exp \left( 2\left\vert
\text{Im}\sqrt{iy}\right\vert d\right) \right) \text{ and }F_{1}\left(
iy\right) =o\left( \exp \left( 2\left\vert \text{Im}\sqrt{iy}\right\vert
d\right) \right)
\end{equation*}%
as $y$ $\left( \text{real}\right) \rightarrow \infty .$ In this case, the
assumption $\left( \ref{hypo8}\right) $ can be replaced by%
\begin{equation*}
N_{W}(t)+N_{W^{\infty }}(t)\geq AN_{\sigma \left( B\right) }(t)+\left( \frac{%
d}{\pi }-A\right) N_{\sigma \left( B^{\infty }\right) }(t)-\frac{A}{2}.
\end{equation*}
\end{remark}

\subsection{Case III: $q$ is known on $\left[ b,\protect\pi \right] ,$ where
$b\in \left( 0,d\right) $}

\subsubsection{Pairs of Eigenvalues and Normalizing Constants}

\begin{theorem}
\label{theorem copy(4)}Assume Hypothesis \ref{hypo1} and suppose that $q,$ $%
\widetilde{q}\in $ $C^{m}$ near $b\in \left( 0,d\right) ,$ $q=\widetilde{q}$
a.e. on $\left[ b,\pi \right] ,$ $\beta =\widetilde{\beta },$ $\gamma =%
\widetilde{\gamma }$ and
\begin{eqnarray}
&&N_{W}(t)+N_{W_{1}}(t)+N_{W^{\infty }}(t)+N_{W_{1}^{\infty }}(t)
\label{hypo9} \\
&\geq &AN_{\sigma \left( B\right) }(t)+\left( \frac{2b}{\pi }-A\right)
N_{\sigma \left( B^{\infty }\right) }(t)-\frac{A}{2}-\frac{m+1}{2}  \notag
\end{eqnarray}%
for sufficiently large $t\in
\mathbb{R}
.$ Then $h=\widetilde{h}$ and $q=\widetilde{q}$ a.e. on $\left[ 0,\pi \right]
.$
\end{theorem}

\begin{proof}
Denote
\begin{equation*}
H\left( \lambda \right) :=\frac{F\left( \lambda \right) }{G_{\Xi }\left(
\lambda \right) },
\end{equation*}%
where $G_{\Xi }\left( \lambda \right) $ is similarly defined as in $\left( %
\ref{G3}\right) $ and $F\left( \lambda \right) $ is defined by $\left( \ref%
{FFF}\right) .$ Then it follows from Lemma \ref{Fqh} that if $\beta =%
\widetilde{\beta },$ $\gamma =\widetilde{\gamma },$
\begin{equation}
F\left( \lambda \right) =\left. \left\langle \varphi \left( x,\lambda
\right) ,\widetilde{\varphi }\left( x,\lambda \right) \right\rangle
\right\vert _{x=b}.\text{ }  \label{FLAMUDA3}
\end{equation}%
In addition, if $q=\tilde{q}$ a.e. on $\left[ b,\pi \right] ,$ and $q,$ $%
\tilde{q}\in C^{m}$ near $b\in \left( 0,d\right) ,$ one observes from $%
\left( \ref{faifai}\right) $ and Proposition \ref{ooo copy(1)} that
\begin{equation}
\left\vert F\left( iy\right) \right\vert =o\left( \left\vert y\right\vert ^{-%
\frac{m+1}{2}}\exp \left( 2\left\vert \text{Im}\sqrt{iy}\right\vert b\right)
\right) \text{ as }y\text{ }\left( \text{real}\right) \rightarrow \infty
\text{.}  \label{FIY3}
\end{equation}%
By Lemma \ref{number} and $\left( \ref{hypo9}\right) ,$ we have
\begin{equation*}
\left\vert G_{\Xi }\left( iy\right) \right\vert \geq M\left\vert
y\right\vert ^{-\frac{m+1}{2}}\exp \left( 2\left\vert \text{Im}\sqrt{iy}%
\right\vert b\right) .
\end{equation*}%
Therefore,%
\begin{equation*}
\left\vert H\left( iy\right) \right\vert \leq \left\vert \frac{o\left(
\left\vert y\right\vert ^{-\frac{m+1}{2}}\exp \left( 2\left\vert \text{Im}%
\sqrt{iy}\right\vert b\right) \right) }{M\left\vert y\right\vert ^{-\frac{m+1%
}{2}}\exp \left( 2\left\vert \text{Im}\sqrt{iy}\right\vert b\right) }%
\right\vert =o\left( 1\right) \text{ as }y\text{ }\left( \text{real}\right)
\rightarrow \infty .
\end{equation*}%
This implies that $H\left( \lambda \right) \equiv 0\ $and thus $F\left(
\lambda \right) \equiv 0$ for all $\lambda \in
\mathbb{C}
$ by the argument of the proof of Theorem \ref{theorem}$.$ Then we conclude
the statement of this theorem from Lemma \ref{unique by F}.
\end{proof}

\subsubsection{Pairs of Eigenvalues and Ratios}

\begin{theorem}
\label{theorem copy(5)}Assume Hypothesis \ref{hypo2} and suppose that $q=%
\widetilde{q}$ a.e. on $\left[ d,\pi \right] ,$ $\beta =\widetilde{\beta },$
$\gamma =\widetilde{\gamma }$ and%
\begin{equation}
N_{W}(t)+N_{W^{\infty }}(t)\geq AN_{\sigma \left( B\right) }(t)+\left( \frac{%
b}{\pi }-A\right) N_{\sigma \left( B^{\infty }\right) }(t)-\frac{A}{2}
\label{hypo10}
\end{equation}%
for sufficiently large $t\in
\mathbb{R}
,$ where $\epsilon $ is an arbitrary positive constant.$\ $Then $h=%
\widetilde{h}$ and $q=\widetilde{q}$ a.e. on $\left[ 0,\pi \right] .$
\end{theorem}

\begin{proof}
Denote%
\begin{equation}
H_{1}\left( \lambda \right) :=\frac{F_{1}\left( \lambda \right) }{G_{\Theta
}\left( \lambda \right) },\text{ }H_{2}\left( \lambda \right) =\frac{%
F_{2}\left( \lambda \right) }{G_{\Theta }\left( \lambda \right) },
\label{H2}
\end{equation}%
where $G_{\Theta }\left( \lambda \right) :=G_{W}\left( \lambda \right)
G_{W^{\infty }}\left( \lambda \right) ,$
\begin{equation*}
G_{W}\left( \lambda \right) :=\prod\limits_{\lambda _{n}\in \widehat{W},n\in
S_{B}}\left( 1-\frac{\lambda }{\lambda _{n}}\right) ^{k_{n}},\text{ }%
G_{W^{\infty }}\left( \lambda \right) :=\prod\limits_{\lambda _{n}^{\infty
}\in \widehat{W^{\infty }},n\in S_{B^{\infty }}}\left( 1-\frac{\lambda }{%
\lambda _{n}^{\infty }}\right) ^{k_{n}^{\infty }}
\end{equation*}%
and%
\begin{equation*}
F_{1}\left( \lambda \right) :=\varphi \left( b,\lambda \right) -\widetilde{%
\varphi }\left( b,\lambda \right) ,\text{ }F_{2}\left( \lambda \right)
:=\varphi ^{\prime }\left( b,\lambda \right) -\widetilde{\varphi }^{\prime
}\left( b,\lambda \right) .
\end{equation*}%
By a similar method to that of Theorem \ref{theorem copy(1)}, one can easily
deduce that $H_{1}\left( \lambda \right) $ and $H_{2}\left( \lambda \right) $
are entire functions of order less than $\frac{1}{2}$ from the facts $\left( %
\ref{ratios}\right) ,$ $\left( \ref{7}\right) ,$ $\left( \ref{3f}\right) ,$ $%
\left( \ref{hhypo 3}\right) $, $H=\widetilde{H}$, $\beta =\widetilde{\beta }%
, $ $\gamma =\widetilde{\gamma }$ and $q=\widetilde{q}$ a.e. on $\left[
b,\pi \right] .$

In view of Lemma \ref{number} and $\left( \ref{hypo10}\right) ,$ one has
\begin{equation}
\left\vert G_{\Theta }\left( iy\right) \right\vert \geq M\exp \left(
\left\vert \text{Im}\sqrt{iy}\right\vert b\right) .  \label{GTHETE}
\end{equation}%
By $\left( \ref{1}\right) $ we also infer that%
\begin{equation*}
\left\vert F_{1}\left( iy\right) \right\vert =O\left( \left\vert
y\right\vert ^{-\frac{1}{2}}\exp \left( \left\vert \text{Im}\sqrt{iy}%
\right\vert b\right) \right) \text{ as }y\text{ }\left( \text{real}\right)
\rightarrow \infty .
\end{equation*}%
Therefore,
\begin{equation*}
\left\vert H_{1}\left( iy\right) \right\vert \leq \left\vert \frac{O\left(
\left\vert y\right\vert ^{-\frac{1}{2}}\exp \left( \left\vert \text{Im}\sqrt{%
iy}\right\vert b\right) \right) }{M\exp \left( \left\vert \text{Im}\sqrt{iy}%
\right\vert b\right) }\right\vert =O\left( y^{-\frac{1}{2}}\right) \text{ }
\end{equation*}%
as $y$ $\left( \text{real}\right) \rightarrow \infty .$ Now by Lemma \ref%
{proposition}$,$ we can obtain that $H_{1}\left( \lambda \right) \equiv 0,$
i.e., $\varphi \left( b,\lambda \right) \equiv \widetilde{\varphi }\left(
b,\lambda \right) $ for all $\lambda \in
\mathbb{C}
$. Then it follows from $\left( \ref{FLAMUDA3}\right) $ and $\left( \ref{H2}%
\right) $ that
\begin{equation*}
H_{2}\left( \lambda \right) =\frac{\left[ \varphi ^{\prime }\left( b,\lambda
\right) -\widetilde{\varphi }^{\prime }\left( b,\lambda \right) \right]
\varphi \left( b,\lambda \right) }{G_{\Theta }\left( \lambda \right) \varphi
\left( b,\lambda \right) }=\frac{-F\left( \lambda \right) }{G_{\Theta
}\left( \lambda \right) \varphi \left( b,\lambda \right) }.
\end{equation*}%
Thus by $\left( \ref{faib}\right) ,$ $\left( \ref{FIY3}\right) $ $\left(
\text{for }m=-1\right) $ and $\left( \ref{GTHETE}\right) ,$ we have
\begin{equation*}
\left\vert H_{2}\left( iy\right) \right\vert \leq \left\vert \frac{o\left(
\exp \left( 2\left\vert \text{Im}\sqrt{iy}\right\vert b\right) \right) }{%
M\exp \left( \left\vert \text{Im}\sqrt{iy}\right\vert b\right) \frac{1}{2}%
\exp \left( \left\vert \text{Im}\sqrt{iy}\right\vert b\right) \left(
1+o\left( 1\right) \right) }\right\vert =o\left( 1\right) .
\end{equation*}%
Then by Lemma \ref{proposition} we infer that $H_{2}\left( \lambda \right)
\equiv 0$ and then $F\left( \lambda \right) \equiv 0$ for all $\lambda \in
\mathbb{C}
.$ Now we can conclude from Lemma \ref{unique by F} that $h=\widetilde{h}$
and $q=\widetilde{q}$ a.e. on $\left[ 0,\pi \right] .$ The proof is thus
completed.
\end{proof}




\section*{Appendix}

For the self-adjoint classical Sturm-Liouville operators, an interesting
uniqueness result is to assume that the potential $q$ satisfies a local
smoothness condition so that some eigenvalues and norming constants can be
missing. While in \cite{ges2,wei,zhaoying} the key technique relies on the
high-energy asymptotic expansion of the Weyl $m$-function \cite{weyl}, in
our non-self-adjoint setting, the key to prove the uniqueness problems
(Theorem \ref{theorem}, Theorem \ref{theorem copy(4)}, Remark $\ref{theorem
copy(6)},$ Corollary \ref{corollary copy(1)}--\ref{corollary copy(3)}) will
be Proposition \ref{ooo copy(1)}, to be established below.

\begin{definition}
\label{yyy}For $i=1,2,$ let $y_{i,r}(x,\lambda )$ and $\widetilde{y}%
_{i,r}(x,\lambda )$ be solutions of $\left( \ref{ly}\right) $ corresponding
to the potential $q$ and $\widetilde{q},$ respectively, where $%
y_{i,r}(x,\lambda )\ $and $\widetilde{y}_{i,r}(x,\lambda )$ satisfy the
initial conditions
\begin{eqnarray*}
y_{1,r}(r,\lambda ) &=&y_{2,r}^{\prime }(r,\lambda )=1,\text{ }%
y_{2,r}(r,\lambda )=y_{1,r}^{\prime }(r,\lambda )=0,\  \\
\widetilde{y}_{1,r}(r,\lambda ) &=&\widetilde{y}_{2,r}^{\prime }(r,\lambda
)=1,\text{ }\widetilde{y}_{2,r}(r,\lambda )=\widetilde{y}_{1,r}^{\prime
}(r,\lambda )=0,\ r\in \left[ 0,\pi \right) .
\end{eqnarray*}%
For simplicity, denote $y_{1}(x,\lambda ):=y_{1,0}(x,\lambda ),$ $%
y_{2}(x,\lambda ):=y_{2,0}(x,\lambda ),$ $\widetilde{y}_{1}(x,\lambda ):=%
\widetilde{y}_{1,0}(x,\lambda ),$ $\widetilde{y}_{2}(x,\lambda ):=\widetilde{%
y}_{2,0}(x,\lambda ).$
\end{definition}

\begin{proposition}
\label{ooo copy(1)}Let $x_{0}\in \left( r,\pi \right] \ $where $r\in \left[
0,\pi \right) $ and assume that $q,$ $\widetilde{q}\in $ $C^{m}\left[
x_{0}-\delta ,x_{0}\right] \ $for some sufficiently small $\delta >0$ and
some $m\in
\mathbb{N}
_{0}.$ If $q_{-}^{(j)}(x_{0})=\widetilde{q}_{-}^{(j)}(x_{0})$ for $%
j=0,1,\ldots ,m,$ then
\begin{eqnarray}
&&  \label{1111} \\
y_{1,r}(x_{0},\lambda )\widetilde{y}_{1,r}^{\prime }(x_{0},\lambda
)-y_{1,r}^{\prime }(x_{0},\lambda )\widetilde{y}_{1,r}(x_{0},\lambda )
&=&o\left( \frac{\exp \left( 2\left\vert \text{Im}\sqrt{\lambda }\right\vert
\left( x_{0}-r\right) \right) }{\left\vert \sqrt{\lambda }\right\vert ^{m+1}}%
\right) ,  \notag \\
&&  \label{1112} \\
y_{1,r}(x_{0},\lambda )\widetilde{y}_{2,r}^{\prime }(x_{0},\lambda
)-y_{1,r}^{\prime }(x_{0},\lambda )\widetilde{y}_{2,r}(x_{0},\lambda )
&=&o\left( \frac{\exp \left( 2\left\vert \text{Im}\sqrt{\lambda }\right\vert
\left( x_{0}-r\right) \right) }{\left\vert \sqrt{\lambda }\right\vert ^{m+2}}%
\right) ,  \notag \\
&&  \label{1113} \\
\widetilde{y}_{1,r}^{\prime }(x_{0},\lambda )y_{2,r}(x_{0},\lambda )-%
\widetilde{y}_{1,r}(x_{0},\lambda )y_{2,r}^{\prime }(x_{0},\lambda )
&=&o\left( \frac{\exp \left( 2\left\vert \text{Im}\sqrt{\lambda }\right\vert
\left( x_{0}-r\right) \right) }{\left\vert \sqrt{\lambda }\right\vert ^{m+2}}%
\right) ,  \notag \\
&&  \label{1114} \\
y_{2,r}(x_{0},\lambda )\widetilde{y}_{2,r}^{\prime }(x_{0},\lambda
)-y_{2,r}^{\prime }(x_{0},\lambda )\widetilde{y}_{2,r}(x_{0},\lambda )
&=&o\left( \frac{\exp \left( 2\left\vert \text{Im}\sqrt{\lambda }\right\vert
\left( x_{0}-r\right) \right) }{\left\vert \sqrt{\lambda }\right\vert ^{m+3}}%
\right)  \notag
\end{eqnarray}%
as $\left\vert \lambda \right\vert \rightarrow \infty $ in $\Lambda _{\zeta
}:=\left\{ \lambda \in
\mathbb{C}
:\zeta <\text{Arg}\left( \lambda \right) <\pi -\zeta \text{ for }\zeta
>0\right\} .$
\end{proposition}

\begin{remark}
For $f\in $ $C^{m}\left[ x_{0}-\delta ,x_{0}\right] ,$ we adopt following
notations in this section: $\ $
\begin{eqnarray*}
f_{-}^{(0)}(x_{0}) &:&=f(x_{0}),\text{ }f_{-}^{(1)}(x_{0}):=\lim\limits_{x%
\rightarrow x_{0}^{-}}\frac{f(x)-f(x_{0})}{x-x_{0}}, \\
f_{-}^{(j)}(x_{0}) &:&=\lim\limits_{x\rightarrow x_{0}^{-}}\frac{%
f^{(j-1)}(x)-f_{-}^{(j-1)}(x_{0})}{x-x_{0}}\text{ for }j=2,3,\ldots ,m.
\end{eqnarray*}%
In addition, $f\in $ $C^{m}\left[ x_{0}-\delta ,x_{0}\right] \ $implies $%
\lim\limits_{x\rightarrow x_{0}^{-}}f^{(j)}(x)=f_{-}^{(j)}(x_{0})$ for $%
j=0,1,\ldots ,m.$
\end{remark}

The proof of Proposition \ref{ooo copy(1)} will be given at the end of this
appendix after the proof of the following lemma.

\begin{lemma}
\label{ooo}Let $x_{0}\in \left( 0,\pi \right] $ and $q,$ $\widetilde{q}$ $%
\in $ $C^{m}\left[ 0,x_{0}\right] \ $for some $m\in
\mathbb{N}
_{0}.$ If
\begin{equation}
q_{-}^{(j)}(x_{0})=\widetilde{q}_{-}^{(j)}(x_{0})  \label{qq}
\end{equation}%
for $j=0,1,\ldots ,m,$ then
\begin{eqnarray}
y_{2}(x_{0},\lambda )\widetilde{y}_{2}^{\prime }(x_{0},\lambda
)-y_{2}^{\prime }(x_{0},\lambda )\widetilde{y}_{2}(x_{0},\lambda )
&=&o\left( \frac{\exp \left( 2\left\vert \text{Im}\sqrt{\lambda }\right\vert
x_{0}\right) }{\left\vert \sqrt{\lambda }\right\vert ^{m+3}}\right) ,
\label{4} \\
y_{1}(x_{0},\lambda )\widetilde{y}_{1}^{\prime }(x_{0},\lambda
)-y_{1}^{\prime }(x_{0},\lambda )\widetilde{y}_{1}(x_{0},\lambda )
&=&o\left( \frac{\exp \left( 2\left\vert \text{Im}\sqrt{\lambda }\right\vert
x_{0}\right) }{\left( \sqrt{\lambda }\right) ^{m+1}}\right) , \\
y_{1}(x_{0},\lambda )\widetilde{y}_{2}^{\prime }(x_{0},\lambda
)-y_{1}^{\prime }(x_{0},\lambda )\widetilde{y}_{2}(x_{0},\lambda )
&=&o\left( \frac{\exp \left( 2\left\vert \text{Im}\sqrt{\lambda }\right\vert
x_{0}\right) }{\left\vert \sqrt{\lambda }\right\vert ^{m+2}}\right) , \\
\widetilde{y}_{1}^{\prime }(x_{0},\lambda )y_{2}(x_{0},\lambda )-\widetilde{y%
}_{1}(x_{0},\lambda )y_{2}^{\prime }(x_{0},\lambda ) &=&o\left( \frac{\exp
\left( 2\left\vert \text{Im}\sqrt{\lambda }\right\vert x_{0}\right) }{%
\left\vert \sqrt{\lambda }\right\vert ^{m+2}}\right)
\end{eqnarray}%
as $\left\vert \lambda \right\vert \rightarrow \infty $ in the sector $%
\Lambda _{\zeta }.$
\end{lemma}

We shall prove Lemma \ref{ooo} by analyzing the asymptotic expansion of the
fundamental solutions $\left( \text{see Lemma \ref{y2} and Lemma }\ref{yy2}%
\right) $. Now we first give some preliminary facts and notations.

Recall the solution $y_{2}$ defined by Definition \ref{yyy}, then it follows
from \cite{inverse} that
\begin{equation}
y_{2}(x,\lambda )=\sum\limits_{p=0}^{\infty }S_{p}(x,\lambda ),\text{ }%
y_{2}^{\prime }(x,\lambda )=\sum\limits_{p=0}^{\infty }C_{p}(x,\lambda ),
\label{3.43}
\end{equation}%
where $S_{0}(x,\lambda )=\frac{\sin (\sqrt{\lambda }x)}{\sqrt{\lambda }}$, $%
C_{0}(x,\lambda )=\cos \left( \sqrt{\lambda }x\right) ,\ $and for $p\geq 1,$
\begin{eqnarray}
S_{p}(x,\lambda ) &=&\int_{0}^{x}\frac{\sin (\sqrt{\lambda }\left(
x-t\right) )}{\sqrt{\lambda }}q(t)S_{p-1}(t,\lambda )\mathrm{d}t,\text{ }
\label{4.1} \\
C_{p}(x,\lambda ) &=&\int_{0}^{x}\cos \left( \sqrt{\lambda }\left(
x-t\right) \right) q(t)S_{p-1}(t,\lambda )\mathrm{d}t\text{.}  \label{4.2}
\end{eqnarray}%
In what follows, we adopt the following notations:%
\begin{equation*}
\left( \pm \right) _{j}=\left\{
\begin{array}{c}
-1\text{ if }j=4s,4s+1, \\
1\text{ if }j=4s+2,4s+3,%
\end{array}%
\right.
\end{equation*}%
and%
\begin{equation*}
\nu _{2s}(x,\lambda ):=\frac{\sin (\sqrt{\lambda }x)}{\left( 2\sqrt{\lambda }%
\right) ^{2s}},\text{ }\nu _{2s+1}(x,\lambda ):=\frac{\cos (\sqrt{\lambda }x)%
}{\left( 2\sqrt{\lambda }\right) ^{2s+1}},\text{ }s\in
\mathbb{N}
_{0}.
\end{equation*}%
Then we have the following statement relating to $S_{p}$ defined by $\left( %
\ref{4.1}\right) .$

\begin{lemma}
\label{y2}Assume that $q$ $\in $ $C^{m}\left[ 0,\delta \right] $ for some $%
\delta >0\ $and some $m\in
\mathbb{N}
.$ Denote $\sigma \left( x\right) :=\int_{0}^{x}q(t)\mathrm{d}t.$ Then for $%
x\in \left[ 0,\delta \right] ,$ we have
\begin{eqnarray}
&&S_{1}(x,\lambda )=\sum\limits_{j=1}^{m+1}\frac{\nu _{j}(x,\lambda )}{\sqrt{%
\lambda }}f_{1,j}\left( x\right) +\frac{\left( \pm \right) _{m+2}}{\sqrt{%
\lambda }}\int_{0}^{x}\nu _{m+1}(x-2t,\lambda )q^{\left( m\right) }(t)dt,
\label{4.4} \\
&&S_{2}(x,\lambda )=\sum\limits_{j=1}^{m+2}\frac{\nu _{j}(x,\lambda )}{\sqrt{%
\lambda }}f_{2,j}\left( x\right) +B_{2}\left( x,\lambda \right) ,
\label{4.5} \\
&&S_{p}(x,\lambda )=\sum\limits_{j=1}^{m+2}\frac{\nu _{j}(x,\lambda )}{\sqrt{%
\lambda }}f_{p,j}\left( x\right) +B_{p}\left( x,\lambda \right) \text{ for }%
p=3,\ldots ,m+2  \label{4.6}
\end{eqnarray}%
where%
\begin{eqnarray*}
B_{2}\left( x,\lambda \right) &=&-\frac{\left( \pm \right) _{m+3}}{\sqrt{%
\lambda }}\int_{0}^{x}\nu _{m+2}(x-2t,\lambda )\left(
\sum\limits_{j=1}^{m+1}\left( \pm \right) _{j}\left( q(t)f_{1,j}\left(
t\right) \right) ^{\left( m+1-j\right) }\right) dt \\
&&+\frac{\left( \pm \right) _{m+2}}{\sqrt{\lambda }}\int_{0}^{x}\frac{\sin
\sqrt{\lambda }\left( x-t\right) }{\sqrt{\lambda }}q(t)\int_{0}^{t}\nu
_{m+1}(t-2s,\lambda )q^{\left( m\right) }(s)dsdt, \\
B_{p}\left( x,\lambda \right) &=&-\frac{\left( \pm \right) _{m+3}}{\sqrt{%
\lambda }}\int_{0}^{x}\nu _{m+2}(x-2t,\lambda )\sum\limits_{j=1}^{m+1}\left(
\pm \right) _{j}\left( q(t)f_{p-1,j}(t)\right) ^{\left( m+1-j\right) }\left(
t\right) dt \\
&&+\int_{0}^{x}\frac{\sin (\sqrt{\lambda }\left( x-t\right) )}{\sqrt{\lambda
}}q(t)\left[ \frac{\nu _{m+2}(t,\lambda )}{\sqrt{\lambda }}f_{p-1,m+2}\left(
t\right) +B_{p-1}\left( t,\lambda \right) \right] \mathrm{d}t
\end{eqnarray*}%
for $p=3,\ldots m+2,$ and the functions $f_{p,j}\left( x\right) $ are
defined by the recurrence relations%
\begin{eqnarray*}
\text{ }f_{1,j}\left( x\right) &=&\left( \pm \right) _{j}\left( \sigma
^{\left( j-1\right) }\left( x\right) -(-1)^{j-1}\sigma ^{\left( j-1\right)
}\left( 0\right) \right) ,\text{ } \\
f_{p,p}\left( x\right) &=&(-1)^{p}\int_{0}^{x}q(t)f_{p-1,p-1}\left( t\right)
dt\ \text{for }p=2,\ldots ,m+2, \\
f_{p,j}\left( x\right) &=&-\sum\limits_{s=1}^{j-2}\left( \pm \right)
_{s}\left( \pm \right) _{j}\left( \left( qf_{p-1,s}\right) ^{\left(
j-s-2\right) }\left( x\right) -(-1)^{j-1}\left( qf_{p-1,s}\right) ^{\left(
j-s-2\right) }\left( 0\right) \right) \\
&&+(-1)^{j}\int_{0}^{x}q(t)f_{p-1,j-1}\left( t\right) dt\ \text{for }j>p\
\text{and }p=2,\ldots ,m+2, \\
f_{p,j}\left( x\right) &=&0\text{ for }j<p.
\end{eqnarray*}%
Moreover$,$ $f_{p,j}\in C^{m+p-j+1}\left[ 0,\delta \right] .$
\end{lemma}

\begin{proof}
In order to prove this lemma, we will follow the technique in \cite[Lemma 4.2%
]{savchuk}. We first note that%
\begin{equation}
\frac{\sin (\sqrt{\lambda }\left( x-t\right) )}{\sqrt{\lambda }}\nu
_{j}(t,\lambda )=\left( -1\right) ^{j+1}\nu _{j+1}(x,\lambda )+\nu
_{j+1}(x-2t,\lambda )  \label{4.9}
\end{equation}%
and for $f\in C^{1}\left[ 0,x\right] ,$%
\begin{eqnarray}
&&\int_{0}^{x}\nu _{j}(x-2t,\lambda )f\left( t\right) dt  \label{4.10} \\
&=&\nu _{j+1}(x,\lambda )\left( f\left( x\right) -\left( -1\right)
^{j}f\left( 0\right) \right) +\left( -1\right) ^{j+1}\int_{0}^{x}\nu
_{j+1}(x-2t,\lambda )f^{\prime }\left( t\right) dt.  \notag
\end{eqnarray}%
In view of $\left( \ref{4.9}\right) $ and $\left( \ref{4.10}\right) ,$ one
can easily deduce the expression $\left( \ref{4.4}\right) .$ Now we turn to
deduce the expressions for the other functions $S_{j}.$ Suppose that $%
f_{j}\in C^{m+1-j}\left[ 0,x\right] ,$ then from $\left( \ref{4.9}\right) ,$
we know that for $j=1,\ldots ,m+1,$
\begin{eqnarray}
&&\int_{0}^{x}\frac{\sin (\sqrt{\lambda }\left( x-t\right) )}{\sqrt{\lambda }%
}\nu _{j}(t,\lambda )f_{j}\left( t\right) dt  \label{le} \\
&=&\left( -1\right) ^{j+1}\nu _{j+1}(x,\lambda )\int_{0}^{x}f_{j}\left(
t\right) dt+\int_{0}^{x}\nu _{j+1}(x-2t,\lambda )f_{j}\left( t\right) dt,
\notag
\end{eqnarray}%
Moreover, integrating by parts the second summand on the right-hand side of
the above equality $m+1-j$ times and using $\left( \ref{4.9}\right) ,$ it
follows that for $j=1,\ldots ,m,$
\begin{eqnarray}
&&\int_{0}^{x}\frac{\sin (\sqrt{\lambda }\left( x-t\right) )}{\sqrt{\lambda }%
}\nu _{j}(t,\lambda )f_{j}\left( t\right) dt  \label{le1} \\
&=&\left( -1\right) ^{j+1}\nu _{j+1}(x,\lambda )\int_{0}^{x}f_{j}\left(
t\right) dt-\sum_{s=j+2}^{m+2}\nu _{s}(x,\lambda )\left( \pm \right)
_{j}\left( \pm \right) _{s}(f_{j}^{\left( s-j-2\right) }\left( x\right)
\notag \\
&&-(-1)^{s-1}f_{j}^{\left( s-j-2\right) }\left( 0\right) )-\left( \pm
\right) _{j}\left( \pm \right) _{m+3}\int_{0}^{x}\nu _{m+2}(x-2t,\lambda
)f_{j}^{\left( m+1-j\right) }\left( t\right) dt.  \notag
\end{eqnarray}%
Therefore, by virtue of $\left( \ref{le}\right) $ $($for $j=m+1)$ and $%
\left( \ref{le1}\right) ,$ for $x\in \left( 0,\delta \right] $ we have that
\begin{eqnarray}
&&\sum\limits_{j=1}^{m+1}\int_{0}^{x}\frac{\sin (\sqrt{\lambda }\left(
x-t\right) )}{\sqrt{\lambda }}\nu _{j}(t,\lambda )f_{j}\left( t\right) dt
\label{4.11} \\
&=&\left( -1\right) ^{2}\nu _{2}(x,\lambda )\int_{0}^{x}f_{1}\left( t\right)
dt  \notag \\
&&+\sum_{j=3}^{m+2}\nu _{j}(x,\lambda )(-\sum\limits_{s=1}^{j-2}\left( \pm
\right) _{s}\left( \pm \right) _{j}(f_{s}^{\left( j-s-2\right) }\left(
x\right) -(-1)^{j-1}f_{s}^{\left( j-s-2\right) }\left( 0\right) )  \notag \\
&&+\left( -1\right) ^{j}\int_{0}^{x}f_{j-1}\left( t\right) dt)-\left( \pm
\right) _{m+3}\int_{0}^{x}\nu _{m+2}(x-2t,\lambda
)\sum\limits_{j=1}^{m+1}\left( \pm \right) _{j}f_{j}^{\left( m+1-j\right)
}\left( t\right) dt.  \notag
\end{eqnarray}%
Now in view of $\left( \ref{4.1}\right) $ and $\left( \ref{4.4}\right) ,$ we
obtain that for $x\in \left( 0,\delta \right] ,$%
\begin{eqnarray*}
S_{2}(x,\lambda ) &=&\frac{1}{\sqrt{\lambda }}\sum\limits_{j=1}^{m+1}%
\int_{0}^{x}\frac{\sin (\sqrt{\lambda }\left( x-t\right) )}{\sqrt{\lambda }}%
\nu _{j}(t,\lambda )q(t)f_{1,j}\left( t\right) \mathrm{d}t \\
&&+\frac{\left( \pm \right) _{m+2}}{\sqrt{\lambda }}\int_{0}^{x}\frac{\sin
\sqrt{\lambda }\left( x-t\right) }{\sqrt{\lambda }}q(t)\int_{0}^{t}\nu
_{m+1}(t-2s,\lambda )q^{\left( m\right) }(s)dsdt.
\end{eqnarray*}%
Making use of $\left( \ref{4.11}\right) $ with $f_{j}\left( t\right) $
replaced by $q(t)f_{1,j}\left( t\right) $ and in virtue of the fact $%
qf_{1,j}\in C^{m+1-j}\left[ 0,\delta \right] $ for $j$ $=1,\ldots ,m+1,$ we
obtain $\left( \ref{4.5}\right) .$ Next, from $\left( \ref{4.1}\right) $ and
$\left( \ref{4.5}\right) ,$ it follows that for $x\in \left( 0,\delta \right]
,$%
\begin{eqnarray*}
S_{3}(x,\lambda ) &=&\frac{1}{\sqrt{\lambda }}\sum\limits_{j=1}^{m+1}%
\int_{0}^{x}\frac{\sin (\sqrt{\lambda }\left( x-t\right) )}{\sqrt{\lambda }}%
\nu _{j}(t,\lambda )q(t)f_{2,j}\left( t\right) \mathrm{d}t \\
&&+\int_{0}^{x}\frac{\sin (\sqrt{\lambda }\left( x-t\right) )}{\sqrt{\lambda
}}q(t)\left[ \frac{\nu _{m+2}(t,\lambda )}{\sqrt{\lambda }}f_{2,m+2}\left(
t\right) +B_{2}\left( t,\lambda \right) \right] \mathrm{d}t.
\end{eqnarray*}%
Then the expression $\left( \ref{4.6}\right) $ for $S_{3}$ can be proved by
using $\left( \ref{4.11}\right) $ and letting $f_{j}\left( t\right)
:=q(t)f_{2,j}\left( t\right) .$ The proof of the relation $\left( \ref{4.6}%
\right) $ for $p=4,\ldots ,m+2\ $can be carried out in the same way.
\end{proof}

As a consequence of Lemma \ref{y2}, we have the following assertion relating
to $C_{p}$ defined by $\left( \ref{4.2}\right) .$

\begin{lemma}
\label{yy2}Assume that $q$ $\in $ $C^{m}\left[ 0,\delta \right] $ for some $%
\delta >0\ $and some $m\in
\mathbb{N}
.$ Denote $\sigma \left( x\right) :=\int_{0}^{x}q(t)\mathrm{d}t.$ Then for $%
x\in \left[ 0,\delta \right] ,$ we have%
\begin{eqnarray*}
C_{1}(x,\lambda ) &=&\frac{-f_{1,1}\left( x\right) }{2}\frac{\nu
_{0}(x,\lambda )}{\sqrt{\lambda }}+\sum\limits_{j=1}^{m}\frac{\nu
_{j}(x,\lambda )}{\sqrt{\lambda }}\left[ f_{1,j}^{\prime }\left( x\right) +%
\frac{(-1)^{j+1}f_{1,j+1}\left( x\right) }{2}\right] \\
&&+\frac{\left( \pm \right) _{m+2}}{\sqrt{\lambda }}\int_{0}^{x}\frac{d\nu
_{m+1}(x-2t,\lambda )}{dx}q^{\left( m\right) }(t)dt, \\
C_{2}(x,\lambda ) &=&\sum\limits_{j=1}^{m+1}\frac{\nu _{j}(x,\lambda )}{%
\sqrt{\lambda }}\left[ f_{2,j}^{\prime }\left( x\right) +\frac{%
(-1)^{j+1}f_{2,j+1}\left( x\right) }{2}\right] +D_{2}\left( x,\lambda \right)
\\
C_{p}(x,\lambda ) &=&\sum\limits_{j=1}^{m+1}\frac{\nu _{j}(x,\lambda )}{%
\sqrt{\lambda }}\left[ f_{p,j}^{\prime }\left( x\right) +\frac{%
(-1)^{j+1}f_{p,j+1}\left( x\right) }{2}\right] +D_{p}\left( x,\lambda
\right) \ \text{for }p=3,\ldots ,m+2
\end{eqnarray*}%
where $f_{p,j}\left( x\right) $ are the functions defined in Lemma $\ref{y2}$%
, and%
\begin{eqnarray*}
D_{2}\left( x,\lambda \right) &=&-\frac{\left( \pm \right) _{m+3}}{\sqrt{%
\lambda }}\int_{0}^{x}\frac{d\nu _{m+2}(x-2t,\lambda )}{dx}\left(
\sum\limits_{j=1}^{m+1}\left( \pm \right) _{j}\left( q(t)f_{1,j}\left(
t\right) \right) ^{\left( m+1-j\right) }\right) dt \\
&&+\frac{\left( \pm \right) _{m+2}}{\sqrt{\lambda }}\int_{0}^{x}\cos \sqrt{%
\lambda }\left( x-t\right) q(t)\int_{0}^{t}\nu _{m+1}(t-2s,\lambda
)q^{\left( m\right) }(s)dsdt, \\
D_{p}\left( x,\lambda \right) &=&-\frac{\left( \pm \right) _{m+3}}{\sqrt{%
\lambda }}\int_{0}^{x}\frac{d\nu _{m+2}(x-2t,\lambda )}{dx}%
\sum\limits_{j=1}^{m+1}\left( \pm \right) _{j}\left( qf_{p-1,j}\right)
^{\left( m+1-j\right) }\left( t\right) dt \\
&&+\int_{0}^{x}\cos \sqrt{\lambda }\left( x-t\right) q(t)\left[ \frac{\nu
_{m+2}(t,\lambda )}{\sqrt{\lambda }}f_{p-1,m+2}\left( t\right)
+B_{p-1}\left( t,\lambda \right) \right] \mathrm{d}t
\end{eqnarray*}%
$\ $for $p=3,\ldots ,m+2.$
\end{lemma}

\begin{lemma}
Assume that $q$ $\in $ $C^{m}\left[ 0,\delta \right] $ for some $\delta >0\ $%
and some $m\in
\mathbb{N}
_{0}.$ Then for $x\in \left[ 0,\delta \right] ,$ $y_{2}(x,\lambda )$ and $%
y_{2}^{\prime }(x,\lambda )$ can be rewritten as the following form:%
\begin{eqnarray}
y_{2}(x,\lambda ) &=&\frac{\sin \left( \sqrt{\lambda }x\right) }{\sqrt{%
\lambda }}+\sum\limits_{j=1}^{m+2}a_{j}\left( x\right) \frac{\nu
_{j}(x,\lambda )}{\sqrt{\lambda }}  \label{y} \\
&&+\frac{\left( \pm \right) _{m+2}}{\sqrt{\lambda }}\int_{0}^{x}\nu
_{m+1}(x-2t,\lambda )q^{\left( m\right) }(t)dt  \notag \\
&&+\sum\limits_{p=2}^{m+2}B_{p}\left( x,\lambda \right)
+\sum\limits_{p=m+3}^{\infty }S_{p}\left( x,\lambda \right) ,  \notag
\end{eqnarray}%
and%
\begin{eqnarray}
y_{2}^{\prime }(x,\lambda ) &=&\cos \left( \sqrt{\lambda }x\right)
+\sum\limits_{j=0}^{m+1}b_{j}\left( x\right) \frac{\nu _{j}(x,\lambda )}{%
\sqrt{\lambda }}  \label{y2'} \\
&&+\frac{\left( \pm \right) _{m+2}}{\sqrt{\lambda }}\int_{0}^{x}\frac{d\nu
_{m+1}(x-2t,\lambda )}{dx}q^{\left( m\right) }(t)dt  \notag \\
&&+\sum\limits_{p=2}^{m+2}D_{p}\left( x,\lambda \right)
+\sum\limits_{p=m+3}^{\infty }C_{p}\left( x,\lambda \right) ,  \notag
\end{eqnarray}%
where%
\begin{equation*}
a_{j}\left( x\right) =\sum\limits_{p=1}^{m+2}f_{p,j}\left( x\right) \ \text{%
for }j=1,\ldots m+1,\text{ }a_{m+2}\left( x\right)
=\sum\limits_{p=2}^{m+2}f_{p,m+2}\left( x\right) ,
\end{equation*}%
and%
\begin{eqnarray*}
b_{0}\left( x\right) &=&\frac{-f_{1,1}\left( x\right) }{2},\text{ }%
b_{m+1}\left( x\right) =\sum\limits_{p=2}^{m+2}\left( f_{p,m+1}^{\prime
}\left( x\right) +\frac{(-1)^{m+2}f_{p,m+2}\left( x\right) }{2}\right) , \\
b_{j}\left( x\right) &=&\sum\limits_{p=1}^{m+2}\left( f_{p,j}^{\prime
}\left( x\right) +\frac{(-1)^{j+1}f_{p,j+1}\left( x\right) }{2}\right) ,%
\text{ }j=1,2,\ldots ,m.
\end{eqnarray*}
\end{lemma}

\begin{proof}
For $m\in
\mathbb{N}
,$ the expressions $\left( \ref{y}\right) $ and $\left( \ref{y2'}\right) $
can be directly obtained from $\left( \ref{3.43}\right) ,$ Lemma \ref{y2}
and Lemma \ref{yy2}. For $m=0,$ the proof can be carried out in the same way
even simpler.
\end{proof}

\begin{remark}
\label{yy2 copy(1)}For $g\in L^{1}\left[ 0,x\right] ,$ one notes that the
following identities
\begin{eqnarray}
\int_{0}^{x}\sin \left( \sqrt{\lambda }t\right) g(t)dt &=&o\left( \exp
\left( \left\vert \text{Im}\sqrt{\lambda }\right\vert x\right) \right) ,
\label{by1} \\
\int_{0}^{x}\cos \left( \sqrt{\lambda }t\right) g(t)dt &=&o\left( \exp
\left( \left\vert \text{Im}\sqrt{\lambda }\right\vert x\right) \right)
\label{by2}
\end{eqnarray}%
hold \cite{inverse}. By virtue of $\left( \ref{by1}\right) $ and $\left( \ref%
{by2}\right) $, it is easy to deduce that%
\begin{equation*}
B_{p}\left( x,\lambda \right) =o\left( \frac{\exp \left( \left\vert \text{Im}%
\sqrt{\lambda }\right\vert x\right) }{\left\vert \sqrt{\lambda }\right\vert
^{m+3}}\right) ,\text{ }B_{p}^{\prime }\left( x,\lambda \right) =o\left(
\frac{\exp \left( \left\vert \text{Im}\sqrt{\lambda }\right\vert x\right) }{%
\left\vert \sqrt{\lambda }\right\vert ^{m+2}}\right)
\end{equation*}%
and%
\begin{equation*}
D_{p}\left( x,\lambda \right) =o\left( \frac{\exp \left( \left\vert \text{Im}%
\sqrt{\lambda }\right\vert x\right) }{\left\vert \sqrt{\lambda }\right\vert
^{m+2}}\right) ,\text{ }D_{p}^{\prime }\left( x,\lambda \right) =o\left(
\frac{\exp \left( \left\vert \text{Im}\sqrt{\lambda }\right\vert x\right) }{%
\left\vert \sqrt{\lambda }\right\vert ^{m+1}}\right)
\end{equation*}%
hold for $p=2,3,\ldots ,m+2.$
\end{remark}

\begin{remark}
\label{yy2 copy(3)}Note that%
\begin{eqnarray*}
S_{p}(x,\lambda ) &=&\int_{0\leq t_{1}\leq \cdots \leq
t_{p+1}:=x}\prod\limits_{i=1}^{p}s_{\lambda }(t_{i+1}-t_{i}\,)s_{\lambda
}(t_{1})q(t_{i})\mathrm{d}t_{1}\cdots \,\mathrm{d}t_{p}, \\
C_{p}(x,\lambda ) &=&\int_{0\leq t_{1}\leq \cdots \leq t_{p+1}:=x}c_{\lambda
}(t_{p+1}-t_{p})\prod\limits_{i=1}^{p-1}s_{\lambda
}(t_{i+1}-t_{i}\,)s_{\lambda }(t_{1})q(t_{i})\mathrm{d}t_{1}\cdots \,\mathrm{%
d}t_{p},
\end{eqnarray*}%
where $s_{\lambda }(x):=\frac{\sin (\sqrt{\lambda }x)}{\sqrt{\lambda }}%
,c_{\lambda }(x):=\cos (\sqrt{\lambda }x),$ and%
\begin{equation*}
\left\vert \frac{\sin (\sqrt{\lambda }x)}{\sqrt{\lambda }}\right\vert \leq
\frac{\exp \left( \left\vert \text{Im}\sqrt{\lambda }\right\vert x\right) }{%
\left\vert \sqrt{\lambda }\right\vert },\text{ }\left\vert \cos (\sqrt{%
\lambda }x)\right\vert \leq \exp \left( \left\vert \text{Im}\sqrt{\lambda }%
\right\vert x\right) .
\end{equation*}%
Thus for $\lambda \in
\mathbb{C}
$ and $\left\vert \lambda \right\vert $ being large enough$,$ one has
\begin{eqnarray*}
\left\vert S_{p}(x,\lambda )\right\vert &\leq &\frac{\exp \left( \left\vert
\text{Im}\sqrt{\lambda }\right\vert x\right) }{\left\vert \sqrt{\lambda }%
\right\vert ^{m+4}}\frac{\left( \int_{0}^{x}\left\vert q(t)\right\vert
\mathrm{d}t\right) ^{p}}{p!},\text{ }p\geq m+3, \\
\left\vert C_{p}(x,\lambda )\right\vert &\leq &\frac{\exp \left( \left\vert
\text{Im}\sqrt{\lambda }\right\vert x\right) }{\left\vert \sqrt{\lambda }%
\right\vert ^{m+3}}\frac{\left( \int_{0}^{x}\left\vert q(t)\right\vert
\mathrm{d}t\right) ^{p}}{p!},\text{ }p\geq m+3.
\end{eqnarray*}%
This directly yields that as $\left\vert \lambda \right\vert \rightarrow
\infty ,$
\begin{equation*}
\sum\limits_{p=m+3}^{\infty }S_{p}(x,\lambda )=O\left( \frac{\exp \left(
\left\vert \text{Im}\sqrt{\lambda }\right\vert x\right) }{\left\vert \sqrt{%
\lambda }\right\vert ^{m+4}}\right) ,\text{ }\sum\limits_{p=m+3}^{\infty
}C_{p}(x,\lambda )=O\left( \frac{\exp \left( \left\vert \text{Im}\sqrt{%
\lambda }\right\vert x\right) }{\left\vert \sqrt{\lambda }\right\vert ^{m+3}}%
\right) .
\end{equation*}%
Similarly, one can also obtain that$\ \sum\limits_{p=m+3}^{\infty
}C_{p}^{\prime }(x,\lambda )=O\left( \frac{\exp \left( \left\vert \text{Im}%
\sqrt{\lambda }\right\vert x\right) }{\left\vert \sqrt{\lambda }\right\vert
^{m+2}}\right) $ as $\left\vert \lambda \right\vert \rightarrow \infty .$
\end{remark}

Now we turn to prove Lemma \ref{ooo}.

\begin{proof}[Proof of Lemma \protect\ref{ooo}]
We only aim to prove the relation $\left( \ref{4}\right) ,$ since the other
statements can be treated similarly. We first denote%
\begin{equation*}
g\left( x\right) :=\left\{
\begin{array}{l}
q\left( x\right) ,\text{ }x\in \left[ 0,x_{0}\right] , \\
s\left( x\right) ,\text{ }x\in \left( x_{0},x_{0}+\delta \right] ,%
\end{array}%
\right. \text{ }\widetilde{g}\left( x\right) :=\left\{
\begin{array}{l}
\widetilde{q}\left( x\right) ,\text{ }x\in \left[ 0,x_{0}\right] , \\
s\left( x\right) ,\text{ }x\in \left( x_{0},x_{0}+\delta \right] ,%
\end{array}%
\right.
\end{equation*}%
where $s\left( x\right) =\sum\limits_{j=0}^{m}q_{-}^{(j)}(x_{0})\left(
x-x_{0}\right) ^{j}$ and $\delta $ is some positive constant$.$ Then by $%
\left( \ref{qq}\right) $ it is easy to see that $g,$ $\widetilde{g}$ $\in $ $%
C^{m}\left[ 0,x_{0}+\delta \right] \ $and%
\begin{equation}
g^{\left( j\right) }(x_{0})=q_{-}^{(j)}(x_{0})=\widetilde{g}^{\left(
j\right) }(x_{0})\text{ for }j=0,1,\ldots ,m.  \label{eeeee}
\end{equation}%
For $i=1,2,$ let $w_{2}(x,\lambda )$ and $\widetilde{w}_{2}(x,\lambda )$ be
the fundamental solutions of the equations%
\begin{equation*}
-y^{\prime \prime }+g\left( x\right) y=\lambda y\text{\ and}-y^{\prime
\prime }+\widetilde{g}\left( x\right) y=\lambda y,\text{ }x\in \left(
0,x_{0}+\delta \right)
\end{equation*}%
respectively, where $w_{2}(x,\lambda )$ and $\widetilde{w}_{2}(x,\lambda )$
are determined by the initial conditions%
\begin{equation*}
w_{2}(0,\lambda )=\text{ }\widetilde{w}_{2}(0,\lambda )=0,\text{ }%
w_{2}^{\prime }(0,\lambda )=\widetilde{w}_{2}^{\prime }(0,\lambda )=1.
\end{equation*}%
By $\left( \ref{y}\right) ,$ $\left( \ref{y2'}\right) $, Lemma \ref{y2},
Lemma \ref{yy2}, Remark \ref{yy2 copy(1)} and Remark \ref{yy2 copy(3)}$,$ it
is easy to see that there exist functions $r_{k},u_{k},z_{k}\in C^{1}\left[
0,x_{0}+\delta \right] $ such that for $x\in \left[ 0,x_{0}+\delta \right] ,$
\begin{eqnarray}
&&w_{2}(x,\lambda )\widetilde{w}_{2}^{\prime }(x,\lambda )-w_{2}^{\prime
}(x,\lambda )\widetilde{w}_{2}(x,\lambda )  \notag \\
&=&\sum\limits_{k=0}^{m+3}r_{k}(x)\frac{\sin \left( \sqrt{\lambda }x\right)
\cos \left( \sqrt{\lambda }x\right) }{\left( \sqrt{\lambda }\right) ^{k}}%
+\sum\limits_{k=0}^{m+3}u_{k}(x)\frac{\cos ^{2}\left( \sqrt{\lambda }%
x\right) }{\left( \sqrt{\lambda }\right) ^{k}}  \notag \\
&&+\sum\limits_{k=0}^{m+3}z_{k}(x)\frac{\sin ^{2}\left( \sqrt{\lambda }%
x\right) }{\left( \sqrt{\lambda }\right) ^{k}}+\frac{\left( \pm \right)
_{m+2}}{\sqrt{\lambda }}I_{1}\left( x,\lambda \right) +I_{2}\left( x,\lambda
\right)  \notag \\
&=&\sum\limits_{k=0}^{m+3}\frac{r_{k}(x)}{2}\frac{\sin \left( 2\sqrt{\lambda
}x\right) }{\left( \sqrt{\lambda }\right) ^{k}}+\sum\limits_{k=0}^{m+3}\frac{%
u_{k}(x)-z_{k}(x)}{2}\frac{\cos \left( 2\sqrt{\lambda }x\right) }{\left(
\sqrt{\lambda }\right) ^{k}}  \label{yy'} \\
&&+\sum\limits_{k=0}^{m+3}\frac{u_{k}(x)+z_{k}(x)}{2\left( \sqrt{\lambda }%
\right) ^{k}}+\frac{\left( \pm \right) _{m+2}}{\sqrt{\lambda }}I_{1}\left(
x,\lambda \right) +I_{2}\left( x,\lambda \right) ,  \notag
\end{eqnarray}%
where%
\begin{eqnarray}
I_{1}\left( x,\lambda \right) &=&\frac{\sin \left( \sqrt{\lambda }x\right) }{%
\sqrt{\lambda }}\int_{0}^{x}\frac{d\nu _{m+1}(x-2t,\lambda )}{dx}\left(
\widetilde{g}^{\left( m\right) }(t)-g^{\left( m\right) }(t)\right) dt
\label{definition of I1} \\
&&+\cos \left( \sqrt{\lambda }x\right) \int_{0}^{x}\nu _{m+1}(x-2t,\lambda
)\left( g^{\left( m\right) }(t)-\widetilde{g}^{\left( m\right) }(t)\right) dt
\notag \\
&=&\left\{
\begin{array}{c}
\int_{0}^{x}\frac{\cos \left( 2\sqrt{\lambda }t\right) }{\left( 2\sqrt{%
\lambda }\right) ^{m+1}}\left( g^{\left( m\right) }(t)-\widetilde{g}^{\left(
m\right) }(t)\right) dt\text{ if }m\ \text{is even}, \\
\int_{0}^{x}\frac{\sin \left( 2\sqrt{\lambda }t\right) }{\left( 2\sqrt{%
\lambda }\right) ^{m+1}}\left( \widetilde{g}^{\left( m\right) }(t)-g^{\left(
m\right) }(t)\right) dt\text{ if }m\ \text{is odd},%
\end{array}%
\text{ }\right.  \notag
\end{eqnarray}%
and as $\left\vert \lambda \right\vert \rightarrow \infty ,$
\begin{equation*}
I_{2}\left( x,\lambda \right) =o\left( \frac{\exp \left( 2\left\vert \text{Im%
}\sqrt{\lambda }\right\vert x\right) }{\left\vert \sqrt{\lambda }\right\vert
^{m+3}}\right) ,\text{ }I_{2}^{\prime }\left( x,\lambda \right) =o\left(
\frac{\exp \left( 2\left\vert \text{Im}\sqrt{\lambda }\right\vert x\right) }{%
\left\vert \sqrt{\lambda }\right\vert ^{m+2}}\right) .
\end{equation*}%
In view of $\left( \ref{yy'}\right) $ and the fact $g=\widetilde{g}$ on $%
\left[ x_{0},x_{0}+\delta \right] ,$ one deduces that for $x\in \left[
x_{0},x_{0}+\delta \right] ,$
\begin{eqnarray*}
&&\left( w_{2}(x,\lambda )\widetilde{w}_{2}^{\prime }(x,\lambda
)-w_{2}^{\prime }(x,\lambda )\widetilde{w}_{2}(x,\lambda )\right) ^{\prime }
\\
&=&r_{0}(x)\sqrt{\lambda }\cos \left( 2\sqrt{\lambda }x\right) -\left(
u_{0}(x)-z_{0}(x)\right) \sqrt{\lambda }\sin \left( 2\sqrt{\lambda }x\right)
\\
&&+\sum\limits_{k=0}^{m+2}\frac{u_{k}^{\prime }(x)+z_{k}^{\prime }(x)}{%
2\left( \sqrt{\lambda }\right) ^{k}}+\sum\limits_{k=0}^{m+2}\left(
r_{k+1}(x)+\frac{u_{k}^{\prime }(x)-z_{k}^{\prime }(x)}{2}\right) \frac{\cos
\left( 2\sqrt{\lambda }x\right) }{\left( \sqrt{\lambda }\right) ^{k}} \\
&&-\sum\limits_{k=0}^{m+2}\left( u_{k+1}(x)-z_{k+1}(x)-\frac{r_{k}^{\prime
}(x)}{2}\right) \frac{\sin \left( 2\sqrt{\lambda }x\right) }{\left( \sqrt{%
\lambda }\right) ^{k}}+o\left( \frac{\exp \left( 2\left\vert \text{Im}\sqrt{%
\lambda }\right\vert x\right) }{\left\vert \sqrt{\lambda }\right\vert ^{m+2}}%
\right) \\
&=&0.
\end{eqnarray*}%
This forces that for $x\in \left[ x_{0},x_{0}+\delta \right] ,$ $%
u_{0}(x)-z_{0}(x)=r_{0}(x)\equiv 0,$%
\begin{equation*}
r_{k+1}(x)+\frac{u_{k}^{\prime }(x)-z_{k}^{\prime }(x)}{2}=0,\text{ }%
u_{k+1}(x)-z_{k+1}(x)-\frac{r_{k}^{\prime }(x)}{2}=0\ \text{for }k=0,\ldots
,m+2,
\end{equation*}%
and thus for $k=0,1,\ldots ,m+3$ and $x\in \left[ x_{0},x_{0}+\delta \right]
,$ one has
\begin{equation}
u_{k}(x)-z_{k}(x)=r_{k}(x)\equiv 0.  \label{equal}
\end{equation}%
Therefore, by $\left( \ref{yy'}\right) $ and $\left( \ref{equal}\right) $ we
infer that
\begin{eqnarray}
&&w_{2}(x_{0},\lambda )\widetilde{w}_{2}^{\prime }(x_{0},\lambda
)-w_{2}^{\prime }(x_{0},\lambda )\widetilde{w}_{2}(x_{0},\lambda )
\label{step1} \\
&=&\sum\limits_{k=0}^{m+3}\frac{u_{k}(x_{0})+z_{k}(x_{0})}{2\left( \sqrt{%
\lambda }\right) ^{k}}+\frac{\left( \pm \right) _{m+2}}{\sqrt{\lambda }}%
I_{1}\left( x_{0},\lambda \right) +o\left( \frac{\exp \left( 2\left\vert
\text{Im}\sqrt{\lambda }\right\vert x_{0}\right) }{\left\vert \sqrt{\lambda }%
\right\vert ^{m+3}}\right) .  \notag
\end{eqnarray}

Next, we aim to show that%
\begin{equation}
I_{1}\left( x_{0},\lambda \right) =o\left( \frac{\exp \left( 2\left\vert
\text{Im}\sqrt{\lambda }\right\vert x_{0}\right) }{\left\vert \sqrt{\lambda }%
\right\vert ^{m+2}}\right)   \label{I1}
\end{equation}%
as $\left\vert \lambda \right\vert \rightarrow \infty $ in the sector $%
\Lambda _{\zeta }.$ Due to the definition $\left( \ref{definition of I1}%
\right) $ of $I_{1}\left( x,\lambda \right) ,$ it is sufficient to prove
\begin{eqnarray}
\int_{0}^{x_{0}}\cos \left( 2\sqrt{\lambda }t\right) \left( g^{\left(
m\right) }(t)-\widetilde{g}^{\left( m\right) }(t)\right) dt &=&o\left( \frac{%
\exp \left( 2\left\vert \text{Im}\sqrt{\lambda }\right\vert x_{0}\right) }{%
\left\vert \sqrt{\lambda }\right\vert }\right) ,  \label{cos} \\
\int_{0}^{x_{0}}\sin \left( 2\sqrt{\lambda }t\right) \left( g^{\left(
m\right) }(t)-\widetilde{g}^{\left( m\right) }(t)\right) dt &=&o\left( \frac{%
\exp \left( 2\left\vert \text{Im}\sqrt{\lambda }\right\vert x_{0}\right) }{%
\left\vert \sqrt{\lambda }\right\vert }\right) .  \label{sin}
\end{eqnarray}%
In fact, by $\left( \ref{eeeee}\right) $ $\left( \text{for }j=m\right) $ and
the fact $g,$ $\widetilde{g}$ $\in $ $C^{m}\left[ 0,x_{0}+\delta \right] $
we infer that given any $\epsilon >0,$ there exists a sufficiently small
constant $\delta _{0}>0$ such that%
\begin{equation*}
\left\vert g^{\left( m\right) }(t)-\widetilde{g}^{\left( m\right)
}(t)\right\vert <\epsilon \text{ on }\left[ x_{0}-\delta _{0},x_{0}\right] ,
\end{equation*}%
and thus for $\lambda \in \Lambda _{\zeta }$ and $\left\vert \lambda
\right\vert $ being sufficiently large, we obtain%
\begin{eqnarray*}
&&\left\vert \int_{0}^{x_{0}}\cos \left( 2\sqrt{\lambda }t\right) \left(
g^{\left( m\right) }(t)-\widetilde{g}^{\left( m\right) }(t)\right)
dt\right\vert  \\
&\leq &\left\vert \int_{0}^{x_{0}-\delta _{0}}\cos \left( 2\sqrt{\lambda }%
t\right) \left( g^{\left( m\right) }(t)-\widetilde{g}^{\left( m\right)
}(t)\right) dt\right\vert +\left\vert \int_{x_{0}-\delta _{0}}^{x_{0}}\cos
\left( 2\sqrt{\lambda }t\right) \left( g^{\left( m\right) }(t)-\widetilde{g}%
^{\left( m\right) }(t)\right) dt\right\vert  \\
&\leq &\max_{t\in \left[ 0,x_{0}\right] }\left\vert g^{\left( m\right) }(t)-%
\widetilde{g}^{\left( m\right) }(t)\right\vert \int_{0}^{x_{0}-\delta
_{0}}\left\vert \cos \left( 2\sqrt{\lambda }t\right) \right\vert dt+\epsilon
\int_{x_{0}-\delta _{0}}^{x_{0}}\left\vert \cos \left( 2\sqrt{\lambda }%
t\right) \right\vert dt \\
&\leq &\max_{t\in \left[ 0,x_{0}\right] }\left\vert g^{\left( m\right) }(t)-%
\widetilde{g}^{\left( m\right) }(t)\right\vert \int_{0}^{x_{0}-\delta
_{0}}\exp \left( 2\left\vert \text{Im}\sqrt{\lambda }\right\vert t\right)
dt+\epsilon \int_{x_{0}-\delta _{0}}^{x_{0}}\exp \left( 2\left\vert \text{Im}%
\sqrt{\lambda }\right\vert t\right) dt \\
&\leq &\max_{t\in \left[ 0,x_{0}\right] }\left\vert g^{\left( m\right) }(t)-%
\widetilde{g}^{\left( m\right) }(t)\right\vert \frac{\exp \left( 2\left\vert
\text{Im}\sqrt{\lambda }\right\vert \left( x_{0}-\delta _{0}\right) \right)
}{2\left\vert \text{Im}\sqrt{\lambda }\right\vert }+\epsilon \frac{\exp
\left( 2\left\vert \text{Im}\sqrt{\lambda }\right\vert x_{0}\right) }{%
2\left\vert \text{Im}\sqrt{\lambda }\right\vert } \\
&\leq &\frac{\max\limits_{t\in \left[ 0,x_{0}\right] }\left\vert g^{\left(
m\right) }(t)-\widetilde{g}^{\left( m\right) }(t)\right\vert }{2\exp \left(
2\left\vert \text{Im}\sqrt{\lambda }\right\vert \delta _{0}\right) }\frac{%
\exp \left( 2\left\vert \text{Im}\sqrt{\lambda }\right\vert x_{0}\right) }{%
\left\vert \text{Im}\sqrt{\lambda }\right\vert }+\frac{\epsilon }{2\sin
\frac{\zeta }{2}}\frac{\exp \left( 2\left\vert \text{Im}\sqrt{\lambda }%
\right\vert x_{0}\right) }{\left\vert \sqrt{\lambda }\right\vert } \\
&\leq &\frac{\epsilon }{\sin \frac{\zeta }{2}}\frac{\exp \left( 2\left\vert
\text{Im}\sqrt{\lambda }\right\vert x_{0}\right) }{\left\vert \sqrt{\lambda }%
\right\vert },
\end{eqnarray*}%
where we have used the inequalities
\begin{equation*}
\left\vert \cos \left( 2\sqrt{\lambda }t\right) \right\vert \leq \exp \left(
2\left\vert \text{Im}\sqrt{\lambda }\right\vert \left\vert t\right\vert
\right) \text{ for }\lambda \in
\mathbb{C}
,t\in
\mathbb{R}
\end{equation*}%
and%
\begin{equation}
\left\vert \text{Im}\sqrt{\lambda }\right\vert \geq \left\vert \sqrt{\lambda
}\right\vert \sin \frac{\zeta }{2}\text{ for }\lambda \in \Lambda _{\zeta }.
\label{inequaltiy}
\end{equation}%
This proves the equality $\left( \ref{cos}\right) .$ Note that $\left( \ref%
{sin}\right) $ can be treated similarly, and thus $\left( \ref{I1}\right) $
is proved.

Now by $\left( \ref{step1}\right) $ and $\left( \ref{I1}\right) $ we have
that%
\begin{eqnarray*}
&&w_{2}(x_{0},\lambda )\widetilde{w}_{2}^{\prime }(x_{0},\lambda
)-w_{2}^{\prime }(x_{0},\lambda )\widetilde{w}_{2}(x_{0},\lambda ) \\
&=&\sum\limits_{k=0}^{m+3}\frac{u_{k}(x_{0})+z_{k}(x_{0})}{2\left( \sqrt{%
\lambda }\right) ^{k}}+o\left( \frac{\exp \left( 2\left\vert \text{Im}\sqrt{%
\lambda }\right\vert x_{0}\right) }{\left\vert \sqrt{\lambda }\right\vert
^{m+3}}\right)
\end{eqnarray*}%
as $\left\vert \lambda \right\vert \rightarrow \infty $ in the sector $%
\Lambda _{\zeta }.$ This together with $\left( \ref{inequaltiy}\right) $
directly yields that
\begin{equation*}
w_{2}(x_{0},\lambda )\widetilde{w}_{2}^{\prime }(x_{0},\lambda
)-w_{2}^{\prime }(x_{0},\lambda )\widetilde{w}_{2}(x_{0},\lambda )=o\left(
\frac{\exp \left( 2\left\vert \text{Im}\sqrt{\lambda }\right\vert
x_{0}\right) }{\left\vert \sqrt{\lambda }\right\vert ^{m+3}}\right)
\end{equation*}%
as $\left\vert \lambda \right\vert \rightarrow \infty $ in the sector $%
\Lambda _{\zeta }.$ Now $\left( \ref{4}\right) $ is\ proved, since by the
definition of $g\ $and $\widetilde{g}$ we can infer that
\begin{eqnarray*}
w_{2}(x_{0},\lambda ) &=&y_{2}(x_{0},\lambda ),\text{ }\widetilde{w}%
_{2}(x_{0},\lambda )=\widetilde{y}_{2}(x_{0},\lambda ), \\
w_{2}^{\prime }(x_{0},\lambda ) &=&y_{2}^{\prime }(x_{0},\lambda ),\text{ }%
\widetilde{w}_{2}^{\prime }(x_{0},\lambda )=\widetilde{y}_{2}^{\prime
}(x_{0},\lambda ).
\end{eqnarray*}
\end{proof}

Now we are in a position to prove Proposition \ref{ooo copy(1)}.

\begin{proof}[Proof of Proposition \protect\ref{ooo copy(1)}]
Note that%
\begin{eqnarray}
&&y_{2,r}(x_{0},\lambda )\widetilde{y}_{2,r}^{\prime }(x_{0},\lambda
)-y_{2,r}^{\prime }(x_{0},\lambda )\widetilde{y}_{2,r}(x_{0},\lambda )
\label{yyyy} \\
&=&\left[ y_{2,r}\left( x_{0}-\delta ,\lambda \right) y_{1,x_{0}-\delta
}(x_{0},\lambda )+y_{2,r}^{\prime }\left( x_{0}-\delta ,\lambda \right)
y_{2,x_{0}-\delta }(x_{0},\lambda )\right] \times  \notag \\
&&\left[ \widetilde{y}_{2,r}\left( x_{0}-\delta ,\lambda \right) \widetilde{y%
}_{1,x_{0}-\delta }^{\prime }(x_{0},\lambda )+\widetilde{y}_{2,r}^{\prime
}\left( x_{0}-\delta ,\lambda \right) \widetilde{y}_{2,x_{0}-\delta
}^{\prime }(x_{0},\lambda )\right]  \notag \\
&&-\left[ y_{2,r}\left( x_{0}-\delta ,\lambda \right) y_{1,x_{0}-\delta
}^{\prime }(x_{0},\lambda )+y_{2,r}^{\prime }\left( x_{0}-\delta ,\lambda
\right) y_{2,x_{0}-\delta }^{\prime }(x_{0},\lambda )\right] \times  \notag
\\
&&\left[ \widetilde{y}_{2,r}\left( x_{0}-\delta ,\lambda \right) \widetilde{y%
}_{1,x_{0}-\delta }(x_{0},\lambda )+\widetilde{y}_{2,r}^{\prime }\left(
x_{0}-\delta ,\lambda \right) \widetilde{y}_{2,x_{0}-\delta }(x_{0},\lambda )%
\right]  \notag \\
&=&B_{1}(\lambda )\left[ y_{1,x_{0}-\delta }(x_{0},\lambda )\widetilde{y}%
_{1,x_{0}-\delta }^{\prime }(x_{0},\lambda )-y_{1,x_{0}-\delta }^{\prime
}(x_{0},\lambda )\widetilde{y}_{1,x_{0}-\delta }(x_{0},\lambda )\right]
\notag \\
&&+B_{2}(\lambda )\left[ y_{1,x_{0}-\delta }(x_{0},\lambda )\widetilde{y}%
_{2,x_{0}-\delta }^{\prime }(x_{0},\lambda )-y_{1,x_{0}-\delta }^{\prime
}(x_{0},\lambda )\widetilde{y}_{2,x_{0}-\delta }(x_{0},\lambda )\right]
\notag \\
&&+B_{3}(\lambda )\left[ \widetilde{y}_{1,x_{0}-\delta }^{\prime
}(x_{0},\lambda )y_{2,x_{0}-\delta }(x_{0},\lambda )-\widetilde{y}%
_{1,x_{0}-\delta }(x_{0},\lambda )y_{2,x_{0}-\delta }^{\prime
}(x_{0},\lambda )\right]  \notag \\
&&+B_{4}(\lambda )\left[ y_{2,x_{0}-\delta }(x_{0},\lambda )\widetilde{y}%
_{2,x_{0}-\delta }^{\prime }(x_{0},\lambda )-y_{2,x_{0}-\delta }^{\prime
}(x_{0},\lambda )\widetilde{y}_{2,x_{0}-\delta }(x_{0},\lambda )\right]
\notag
\end{eqnarray}%
where
\begin{eqnarray*}
B_{1}(\lambda ) &=&y_{2,r}\left( x_{0}-\delta ,\lambda \right) \widetilde{y}%
_{2,r}\left( x_{0}-\delta ,\lambda \right) =O\left( \left\vert \lambda
^{-1}\right\vert \exp \left( 2\left\vert \text{Im}\sqrt{\lambda }\right\vert
\left( x_{0}-\delta -r\right) \right) \right) , \\
B_{2}(\lambda ) &=&y_{2,r}\left( x_{0}-\delta ,\lambda \right) \widetilde{y}%
_{2,r}^{\prime }\left( x_{0}-\delta ,\lambda \right) =O\left( \left\vert
\sqrt{\lambda }\right\vert ^{-1}\exp \left( 2\left\vert \text{Im}\sqrt{%
\lambda }\right\vert \left( x_{0}-\delta -r\right) \right) \right) , \\
B_{3}(\lambda ) &=&\widetilde{y}_{2,r}\left( x_{0}-\delta ,\lambda \right)
y_{2,r}^{\prime }\left( x_{0}-\delta ,\lambda \right) =O\left( \left\vert
\sqrt{\lambda }\right\vert ^{-1}\exp \left( 2\left\vert \text{Im}\sqrt{%
\lambda }\right\vert \left( x_{0}-\delta \right) -r\right) \right) , \\
B_{4}(\lambda ) &=&\widetilde{y}_{2,r}^{\prime }\left( x_{0}-\delta ,\lambda
\right) y_{2,r}^{\prime }\left( x_{0}-\delta ,\lambda \right) =O\left( \exp
\left( 2\left\vert \text{Im}\sqrt{\lambda }\right\vert \left( x_{0}-\delta
-r\right) \right) \right) .
\end{eqnarray*}%
The above asymptotics of $B_{1},$ $B_{2},$ $B_{3},$ $B_{4}$ can be obtained
from $\left( \ref{y2y2}\right) .$ Therefore, one can easily deduce from
Lemma \ref{ooo} that
\begin{eqnarray*}
y_{1,x_{0}-\delta }(x_{0},\lambda )\widetilde{y}_{1,x_{0}-\delta }^{\prime
}(x_{0},\lambda )-y_{1,x_{0}-\delta }^{\prime }(x_{0},\lambda )\widetilde{y}%
_{1,x_{0}-\delta }(x_{0},\lambda ) &=&o\left( \frac{\exp \left( 2\left\vert
\text{Im}\sqrt{\lambda }\right\vert \delta \right) }{\left\vert \sqrt{%
\lambda }\right\vert ^{m+1}}\right) , \\
y_{1,x_{0}-\delta }(x_{0},\lambda )\widetilde{y}_{2,x_{0}-\delta }^{\prime
}(x_{0},\lambda )-y_{1,x_{0}-\delta }^{\prime }(x_{0},\lambda )\widetilde{y}%
_{2,x_{0}-\delta }(x_{0},\lambda ) &=&o\left( \frac{\exp \left( 2\left\vert
\text{Im}\sqrt{\lambda }\right\vert \delta \right) }{\left\vert \sqrt{%
\lambda }\right\vert ^{m+2}}\right) , \\
\widetilde{y}_{1,x_{0}-\delta }^{\prime }(x_{0},\lambda )y_{2,x_{0}-\delta
}(x_{0},\lambda )-\widetilde{y}_{1,x_{0}-\delta }(x_{0},\lambda
)y_{2,x_{0}-\delta }^{\prime }(x_{0},\lambda ) &=&o\left( \frac{\exp \left(
2\left\vert \text{Im}\sqrt{\lambda }\right\vert \delta \right) }{\left\vert
\sqrt{\lambda }\right\vert ^{m+2}}\right) , \\
y_{2,x_{0}-\delta }(x_{0},\lambda )\widetilde{y}_{2,x_{0}-\delta }^{\prime
}(x_{0},\lambda )-y_{2,x_{0}-\delta }^{\prime }(x_{0},\lambda )\widetilde{y}%
_{2,x_{0}-\delta }(x_{0},\lambda ) &=&o\left( \frac{\exp \left( 2\left\vert
\text{Im}\sqrt{\lambda }\right\vert \delta \right) }{\left\vert \sqrt{%
\lambda }\right\vert ^{m+3}}\right)
\end{eqnarray*}%
as $\left\vert \lambda \right\vert \rightarrow \infty $ in $\Lambda _{\zeta
}.$ Thus the equality $\left( \ref{1114}\right) $ can be directly obtained
from $\left( \ref{yyyy}\right) .$ The statements $\left( \ref{1111}\right)
-\left( \ref{1113}\right) $ can be proved similarly.
\end{proof}

\begin{remark}
\label{ooo copy(4)}If $q\ $and $\widetilde{q}\ $are both assumed to be in $%
L_{%
\mathbb{C}
}^{1}\left[ 0,\pi \right] ,$ then one can easily find that relations $\left( %
\ref{1111}\right) -\left( \ref{1114}\right) $ still hold by taking $m=-1$.$\
$In fact, in this case, $y_{2,r}(x,\lambda )$ and $y_{2,r}^{\prime
}(x,\lambda )$ have the following asymptotic form \cite{book}:%
\begin{eqnarray}
&&  \label{y2y2} \\
&&y_{2,r}(x,\lambda )  \notag \\
&=&\frac{\sin (\sqrt{\lambda }\left( x-r\right) )}{\sqrt{\lambda }}-Q\left(
x\right) \frac{\cos (\sqrt{\lambda }\left( x-r\right) )}{2\lambda }+o\left(
\frac{\exp \left( \left\vert \text{Im}\sqrt{\lambda }\right\vert \left(
x-r\right) \right) }{\left\vert \lambda \right\vert }\right) ,  \notag \\
&&y_{2,r}^{\prime }(x,\lambda )  \notag \\
&=&\cos (\sqrt{\lambda }\left( x-r\right) )+Q\left( x\right) \frac{\sin (%
\sqrt{\lambda }\left( x-r\right) )}{2\sqrt{\lambda }}+o\left( \frac{\exp
\left( \left\vert \text{Im}\sqrt{\lambda }\right\vert \left( x-r\right)
\right) }{\left\vert \sqrt{\lambda }\right\vert }\right) ,  \notag
\end{eqnarray}%
where $Q\left( x\right) =\int_{r}^{x}q(t)dt.$ Therefore, it is easy to see
that%
\begin{eqnarray*}
&&y_{2,r}(x_{0},\lambda )\widetilde{y}_{2,r}^{\prime }(x_{0},\lambda
)-y_{2,r}^{\prime }(x_{0},\lambda )\widetilde{y}_{2,r}(x_{0},\lambda ) \\
&=&\frac{\int_{r}^{x_{0}}\left( \widetilde{q}(t)-q(t)\right) dt}{2\lambda }%
+o\left( \frac{\exp \left( 2\left\vert \text{Im}\sqrt{\lambda }\right\vert
\left( x_{0}-r\right) \right) }{\left\vert \lambda \right\vert }\right) .
\end{eqnarray*}%
This directly yields $\left( \ref{1114}\right) .$ $\left( \ref{1111}\right)
-\left( \ref{1113}\right) $ can be treated similarly.
\end{remark}

\end{document}